\documentclass[12pt]{amsart} 
\usepackage{verbatim, latexsym, amssymb, amsmath,color}
\usepackage{enumitem}
\usepackage{epsfig}

\begin{document}
\title{On the existence of unstable minimal Heegaard surfaces}

\author{Daniel Ketover}\address{Department of Mathematics\\Princeton
University\\Princeton, NJ 08544}
\author{Yevgeny Liokumovich}
\address{Department of Mathematics\\Massachusetts Institute of Technology\\Cambridge, MA 02139}

\thanks{D.K. was partially supported by an NSF Postdoctoral Research fellowship as well as ERC-2011-StG-278940. Y.L. was partially supported by NSF grant DMS-1711053.}
\email{dketover@math.princeton.edu}
\email{ylio@mit.edu}
\maketitle
\begin{abstract}
We prove that for generic metrics on a $3$-sphere, the minimal surface
obtained from the min-max procedure of Simon-Smith 
has index $1$. 
We prove an analogous result for minimal surfaces arising from strongly irreducible Heegaard sweepouts in
3-manifolds.
We also confirm a conjecture of Pitts-Rubinstein that a strongly irreducible Heegaard splitting in a hyperbolic three-manifold can either be isotoped to a minimal surface of index at most $1$ or else after a neck-pinch is isotopic to a one-sided minimal Heegaard surface.  
\end{abstract}


\def\R{\mathbb R}
\def\RP{\mathbb {RP}}
\def\N{\mathbb N}
\def\Z{\mathbb Z}
\def\C{\mathbb C}
\def\M{\mathcal M}
\def\H{\mathcal H}
\def\L{\mathcal L} 
\def\V{\mathcal V}
\def\Si{\Sigma}
\def\G{\Gamma}

\def\x{\mathbf x}
\def\z{\mathbf z}
\def\g{\mathbf \gamma}
\def\k{\mathbf k}
\def\v{\mathbf v}
\def\r{|x|}
\def\tr{|\partial_r^{\top}|}
\def\nr{|\partial_r^{\bot}|}
\def\dr{|\partial_r|}
\def\vol{\mathrm{vol}}
\def\dmn{\mathrm{dmn}} 
\def\re{\mathrm{Re}}
\def\hess{\mathrm{Hess}\,}
\def\ric{\mathrm{Ric}}
\def\area{\mathrm{area}}
\def\a{\,d\mathcal{H}^n}
\def\b{\,d\mathcal{H}^{n-1}}
\def\d{\mathrm{div}}
\def\dist{\mathrm{dist}}
\def\Ha{\mathcal{H}^2}
\def\Hl{\mathcal{H}^1}

\def\s{\mathrm{spt}}
\def\p{\mathrm{Proj}}
\def\l{\mathrm{loc}}
\newtheorem*{conj}{Conjecture}
\newtheorem*{thma}{Theorem A}
\newtheorem*{thmb}{Theorem B}
\newtheorem*{thmc}{Theorem C}
\newtheorem*{thmd}{Theorem D}
\newtheorem{thm}{Theorem}[section]
\newtheorem{lemma}[thm]{Lemma}
\newtheorem{cor}[thm]{Corollary}
\newtheorem{prop}[thm]{Proposition}
\theoremstyle{remark}
\newtheorem*{rmk}{Remark}
\newtheorem{example}[thm]{Example}
\theoremstyle{definition}
\newtheorem{defi}[thm]{Definition}

\section{Introduction}
The min-max theory was introduced by Almgren in the 60s and then later completed by Pitts in the 80s to construct embedded minimal hypersurfaces in Riemannian manifolds.   Roughly speaking, one considers sweepouts of a manifold and the longest slice in a tightened sweepout that ``pulls over" the entire manifold gives a minimal surface.  Recently, the min-max theory has led to proofs of long-standing problems, for instance the proof of the Willmore Conjecture by Marques and Neves \cite{mn2}.

Almgren-Pitts' approach considers very general sweepouts and it is difficult to control the topology of the minimal surface obtained. In the 80s by Simon-Simon refined Pitts' arguments to allow one to consider sweepouts of a $3$-manifold by surfaces of a fixed topology.  For instance, given $\mathbb{S}^3$ one can consider sweepouts by embedded two-spheres.  Simon and Smith proved that one can work in this restricted class of sweepouts and still obtain a closed embedded minimal surface but with control on the topology.  It was proved in \cite{k} that the topology of the limiting minimal surfaces is achieved roughly speaking after finitely many neck-pinches.  

A basic question is to understand how the Morse index of the minimal surface obtained in either approach is related to the number of parameters used in the construction.  Roughly speaking, a $k$ parameter family should produce an index $k$ critical point.   Suprisingly, the question of estimating the Morse index had been left open since Pitts' original work.   

In the Almgren-Pitts setting, recently Marques-Neves \cite{mn4} made the first advance on this problem by  proving that when the ambient metric is generic (i.e., bumpy in the sense of White \cite{W}) and contains no non-orientable embedded minimal surfaces,  that the support of the minimal surface obtained has index $1$ when one considers one-parameter sweepouts.  In other words, precisely one component is unstable with index $1$ and the other components are all stable.

Under the same hypotheses, we prove in this paper that in the Simon-Smith setting, when running a min-max procedure with two-spheres in $\mathbb{S}^3$, the support of the min-max limit has index $1$.    

The work of Marques-Neves \cite{mn4} involves three components: an upper index bound, a lower index bound, and the fact that the unstable component is achieved with multiplicity $1$.  While the first of these generalizes easily to the Simon-Smith setting, the second and third require new interpolation results.
The main technical contribution of this paper is an interpolation result that rules out
convergence of a min-max sequence to a stable minimal surface.

In this paper, we also confirm a long-standing conjecture of Pitts-Rubinstein: namely to show that in a hyperbolic manifold, if a Heegaard surface is \emph{strongly irreducible} then it can be isotoped to be an index $1$ minimal surface (or else after neck-pinch to the boundary of a twisted interval bundle over a one-sided Heegaard surface).  See Theorem \ref{pr} for a precise statement.

Let us now state our results.  For this we need a number of definitions.

Given a Heegaard splitting $H$ of $M$, a \emph{sweepout by Heegaard surfaces} or \emph{sweepout} is a one parameter family of closed sets $\left\{\Sigma_t\right\}_{t\in [0,1]}$ continuous in the Hausdorff topology such that 
\begin{enumerate}
\setlength{\itemsep}{1pt}
  \setlength{\parskip}{0pt}
  \setlength{\parsep}{0pt}

\item $\Sigma_t$ is an embedded smooth surface isotopic to $H$ for $t\in(0,1)$
\item $\Sigma_t$ varies smoothly for $t\in (0,1)$
\item $\Sigma_{0}$ and $\Sigma_1$ are $1$-d graphs, each one a spine of one of the handlebodies determined by the splitting surface $H$.
\end{enumerate}

If $\Lambda$ is a collection of sweepouts, we say that the set $\Lambda$ is \emph{saturated} if given a map $\phi\in C^{\infty}(I\times M,M)$ such that $\phi(t,-)\in\text{Diff}_{0} M$ for all $t\in I$, and a family  $\left\{\Sigma_t\right\}_{t\in I}\in\Lambda$, we have  $\left\{\phi(t,\Sigma_t)	\right\}_{t\in I}\in\Lambda$. 
Given a Heegaard splitting $H$, let
$\Lambda_H$ denote the set of all sweepouts 
by Heegaard surfaces $\{\Si_t \}$,
such that the corresponding family of mod 2
flat $2$-cycles
is not contractible relative to $\partial [0,1] = \{0,1\}$.
$\Lambda_H$ is a saturated family of sweepouts.

The width associated to $\Lambda_H$ is defined to be
\begin{equation}\label{w}
W(M,\Lambda_H)=\inf_{\left\{\Sigma_t\right\}\in\Lambda}\sup_{t\in I} \mathcal{H}^2(\Sigma_t),
\end{equation}
where $\mathcal{H}^2$ denotes $2$-dimensional Hausdorff measure.  It follows by an easy argument using the isoperimetric inequality (Proposition 1.4 in \cite{cd}) that $W_H>0$.  This expresses the non-triviality of the sweepout.  A \emph{minimizing sequence} is a sequence of families $\left\{\Sigma^n_t\right\}\in\Lambda_H$ such that
\begin{equation} 
\lim_{n\rightarrow\infty}\sup_{t\in[0,1]} \mathcal{H}^2(\Sigma^n_t)=W(M,\Lambda_H).
\end{equation}
\
A \emph{min-max sequence} is then a sequence of slices $\Sigma^n_{t_n}$, $t_n\in (0,1)$ such that
\begin{equation} 
\mathcal{H}^2(\Sigma^n_{t_n})\rightarrow W(M,\Lambda_H).
\end{equation}

The main result due to Simon-Smith is that some min-max sequence converges to a smooth minimal surface realizing the width, whose genus is controlled.  Some genus bounds were proved by Simon-Smith, but the optimal ones quoted below were proved in \cite{k}:

\begin{thm}[Simon-Smith Min-Max Theorem 1982]\label{ss}
Let $M$ be a closed oriented Riemannian $3$-manifold admitting a Heegaard surface $H$ of genus $g$. Then some min-max sequence $\Sigma_{t_i}^i$ of surfaces isotopic to $H$ converges as varifolds to $\sum_{j=1}^k n_j \Gamma_j$, where $\Gamma_j$ are smooth embedded pairwise disjoint minimal surfaces and where $n_j$ are positive integers.  Moreover, \begin{equation}
W(M, \Lambda_H)=\sum_{j=1}^k n_j\mathcal{H}^2(\Gamma_j).
\end{equation} The genus of the limiting minimal surface can be controlled as follows:
\begin{equation}\label{gb}
\sum_{i\in O} n_ig(\Gamma_i) + \frac{1}{2}\sum_{i\in N} n_i(g(\Gamma_i)-1)\leq g,
\end{equation}
where $O$ denotes the set of $i$ such that $\Gamma_i$ is orientable, and $N$ the set of $i$ such that $\Gamma_i$ is non-orientable, and $g(\Gamma)$ denotes the genus of $\Gamma$.  The genus of a non-orientable surface is the number of cross-caps one must attach to a two-sphere to obtain a homeomorphic surface.
\end{thm}

In particular
\begin{thm}[Existence of minimal two-spheres in three-spheres]\label{ss2}
By sweeping out a Riemannian three-sphere by two-spheres, we obtain the existence of a family $\{\Gamma_1,...,\Gamma_k\}$ of pairwise disjoint smooth embedded minimal two-spheres. 
\end{thm}

Marques-Neves \cite{mn4} recently obtained upper index bounds for the min-max minimal surface obtained in Theorem \ref{ss} and \ref{ss2}:
\begin{thm}[Upper Index Bounds]\label{mn}
In the setting of Theorem \ref{ss2} suppose in addition that the metric is bumpy, there holds  
\begin{equation}\label{mn}
\sum_{i=1}^k\text{\emph{index}}(\Gamma_i) \leq 1.
\end{equation}
\end{thm}

\begin{rmk}Recall that a metric is bumpy if no immersed minimal surface contains a non-trivial Jacobi field.  White proved \cite{W} that bumpiness is a generic property for metrics.  In particular, any metric can be perturbed slightly to be bumpy.  
\end{rmk}
Our main result is the following equality in the case of spheres:

\begin{thm}[Index Bounds for Spheres] \label{main}
In the setting of Theorem \ref{ss2}, suppose in addition that the metric is bumpy.  Then the min-max limit satisfies:
\begin{equation}\label{generic}
\sum_{i=1}^k\text{\emph{index}}(\Gamma_i) = 1.
\end{equation}

If the metric is not assumed to be bumpy then we obtain the existence of a minimal surface satisfying \eqref{gb} and 
\begin{equation}\label{nongeneric}
\sum_{i=1}^k\text{\emph{index}}(\Gamma_i) \leq  1 \leq 
\sum_{i=1}^k\text{\emph{index}}(\Gamma_i)+\sum_{i=1}^k\text{\emph{nullity}}(\Gamma_i)
\end{equation}\end{thm}
The index bounds \eqref{generic} and \eqref{nongeneric} were conjectured explicitly by Pitts-Rubinstein \cite{pr2} in 1986.

In particular we have the following:
\begin{thm}\label{indexspheres}
Let $M$ be a Riemannian $3$-manifold diffeomorphic to $\mathbb{S}^3$ endowed with a bumpy metric.  Then $M$ contains an embedded index $1$ minimal two-sphere. 
\end{thm}

A long-standing problem is to prove that in a Riemannian three-sphere $M$, there are at least four embedded minimal two-spheres.  This is the analog of Lusternick Schnirelman's result about the existence of three closed geodesics on two-spheres.   

If $M$ contains a stable
two-sphere, then Theorem \ref{indexspheres} implies the following (by considering the three-balls on each side of this two-sphere):

\begin{thm}[Lusternick--Schnirelman Problem]
Let $M$ be a Riemannian three-sphere containing a stable embedded two-sphere.  Then $M$ contains at least two index one minimal two-spheres.  Thus $M$ contains at least \emph{three} minimal two-spheres.  
\end{thm}

For results in the case when $M$ contains no stable two-spheres, see \cite{HK}.

For strongly irreducible Heegaard splittings, we can use an iterated min-max procedure to obtain:
\begin{thm}[Index Bounds for Minimal Surfaces Arising from Strongly Irreducible Splittings]\label{indexstrong}
Let $\Sigma$ be a strongly irreducible Heegaard splitting surface in a manifold endowed with a bumpy metric.  Then from an iterated min-max procedure we obtain the existence of a family of pairwise disjoint minimal surfaces $\{\Gamma_1,...,\Gamma_k\}$ obtained from $\Sigma$ after neck-pinch surgeries, 
so that 
\begin{equation}\label{generic2}
\sum_{i=1}^k\text{\emph{index}}(\Gamma_i) = 1.
\end{equation}
\end{thm}

Using Theorem \ref{main} together with the Catenoid Estimate \cite{KMN}, we obtain:

\begin{thm}\label{rp3}
Let $M$ be a Riemannian $3$-manifold diffeomorphic to $\mathbb{RP}^3$ endowed with a bumpy metric.  Then $M$ contains a minimal index $1$ two-sphere or minimal index $1$ torus.  
 \end{thm}
 \begin{rmk}
 In the special case that $\mathbb{RP}^3$ is endowed with a metric of positive Ricci curvature, it was proved in \cite{KMN} that it contains a minimal index $1$ torus.  
 \end{rmk}

In this paper, we also confirm a long-standing conjecture of Pitts-Rubinstein \cite{R1} in hyperbolic manifolds:

 \begin{thm}[Pitts-Rubinstein Conjecture (1986)] \label{pr}
 Let $M$ be a hyperbolic $3$-manifold and $\Sigma$ a strongly irreducible Heegaard surface.  Then either 
 \begin{enumerate}
 \item $\Sigma$ is isotopic to a minimal surface of index $1$ or $0$ or 
 \item after a neck-pinch performed on $\Sigma$, the resulting surface is isotopic to the boundary of a tubular neighborhood of a stable one sided Heegaard surface.
 \end{enumerate}
 If $M$ is endowed with a bumpy metric, in case (1) we can assume the index of $\Sigma$ is $1$.
  \end{thm}
  
 \begin{rmk}
 Recall that a one-sided Heegaard surface $\Sigma$ embedded in $M$ is a non-orientable surface such that $M\setminus\Sigma$ is an open handlebody.  An example is $\mathbb{RP}^2\subset\mathbb{RP}^3$ as $\mathbb{RP}^3\setminus\mathbb{RP}^2$ is a three-ball.
 \end{rmk}
 
 A Heegaard splitting is \emph{strongly irreducible} if every curve on $\Sigma$ bounding an essential disk in $H_1$ intersects every such curve bounding an essential disk in $H_2$.  Strongly irreducible Heegaard splittings were first introduced by Casson-Gordon \cite{CG1}, who proved that in non-Haken $3$-manifolds, any splitting can be reduced until it is strongly irreducible.  Thus lowest genus Heegaard splittings in any spherical space form are strongly irreducible.  

Even though they are not hyperbolic manifolds, we still obtain
\begin{thm}[Heegaard tori in lens spaces]
Let $L(p,q)$ be a lens space other than $\mathbb{RP}^3$.   Then $L(p,q)$ contains an index $1$ or $0$ Heegaard torus.  If the metric is assumed bumpy, then the index can be assumed to be $1$.
\end{thm}



As sketched by Rubinstein \cite{R1}, Theorem \ref{pr} gives a minimal surface proof of Waldhausen's conjecture:
\begin{thm}\label{waldhausen}
Let $M$ be a non-Haken hyperbolic $3$-manifold.  Then $M$ contains finitely many irreducible Heegaard splittings of any given genus $g$.
\end{thm}

Waldhausen's conjecture was proved by Tao Li (\cite{Li1},\cite{Li2}, \cite{Li3}) using the combinatorial analog of index $1$ minimal surfaces -- \emph{almost normal} surfaces.  For effective versions of Theorem \ref{waldhausen}, see \cite{CG}, \cite{CGK}.



The organization of this paper is as follows. In Section \ref{interpolation} we explain the main ideas and difficulties in our Interpolation Theorem, which roughly speaking allows us to interpolate between a surface $\Gamma$ close as varifolds to a union of strictly stable minimal surface with integer multiplicities ($\Sigma=\sum n_i\Sigma_i$) and something canonical.   In Section \ref{onecomponent} we consider the case that $\Sigma$ is connected, and show how to isotope $\Gamma$ to a union of several normal graphs over $\Sigma$ joined by necks.  In Section \ref{sec:cables} we describe the notion of ``root sliding" which is useful for global deformations.   In Section \ref{opennecks} we introduce the Light Bulb Theorem and its generalizations which enable us to find necks to further reduce the number of graphs of $\Gamma$ over $\Sigma$.  In Section \ref{multiplecomponents} we generalize to the setting when $\Sigma$ consists of several components.   In Section \ref{deformation1} we apply our interpolation result to obtain the lower index bound.  In Section \ref{applications} we use the index bounds, together with some observations regarding nested minimal surfaces and a characterization of minimal surfaces bounding small volumes to prove the 
conjecture of Pitts-Rubinstein.

\begin{rmk}
During the preparation of this article Antoine Song \cite{So} obtained some related results.
\end{rmk}

{\bf Acknowledgements:}
We would like to thank Andre Neves and Fernando C. Marques for their encouragement and several discussions.  We thank Francesco Lin for some topological advice.  D.K. would like to thank Toby Colding for suggesting this problem.  We are grateful to Dave Gabai for several conversations.

\section{Interpolation}\label{interpolation}

In the proof of the lower index bound (Theorem \ref{main}) to rule out obtaining a stable surface with multiplicity, we must deform slices of a sweepout that come near such a configuration. To that end, the main technical tool is to deform a sequence close in the flat topology to a stable minimal surface with multiplicity to something canonical.

\subsection{Marques-Neves squeezing map.} \label{sec:squeezing}
Let $\Si \subset M$ be a smooth two-sided surface
and let $\exp_{\Si}: \Sigma \times [-h,h] \rightarrow M$
denote the normal exponential map. 
Let $N_{\varepsilon}(S) = \exp_{\Si}(\Sigma \times [-\varepsilon,\varepsilon])$ denote an open $\varepsilon$-tubular neighbourhood
of submanifold $S \subset M$.
It will be convenient for the purposes of this paper
to foliate an open neighbourhood of $\Si$
not by level sets of the distance function,
but rather by hypersurfaces with mean curvature vector pointing
towards $\Si$, which arise as graphs 
of the first eigenfunction
of the stability operator over $\Si$.

Such a foliation gives rise to a diffeomorphism
$\phi: \Si \times (-1,1) \rightarrow \Omega_1 \subset N_h(\Sigma)$,
a collection of open neighbourhoods $\Omega_r = \phi(\Sigma \times (-r,r))$
and squeezing maps $P_t(\phi(x,s)) = \phi(x,(1-t)s)$.
Let $P: \Omega_1 \rightarrow \Sigma$ denote the projection map
$P(\phi(x,s)) = x$.
We refer to \cite[5.7]{mn4} for the details of this
construction. We summarize properties of the map $P_t$:


\begin{enumerate}
\item $P_0(x) = x$ for all $x \in  \Omega_1$
and $P_t(x)=x$ for all $x \in \Sigma$ and $0 \leq t< 1$;
\item There exists $h_0>0$, such that $N_{h_0} \subset \Omega_1$ and
for all positive $h<h_0$ there exists
$t(h) \in (0,1)$  with  $P_{t(h)}(N_{h_0})\subset N_{h}$;
\item 
For any surface $S \subset \Omega_1$ and for all 
$t \in [0,1)$ we have 
$Area(P_t(S)) \leq Area(S)$ with equality holding
if and only if $S \subset \Sigma$;

\item 
Let $U \subset \Sigma$ be an open set,
$f:U \rightarrow \mathbb{R}$ be a smooth function with absolute value
bounded by $h_0$ and let $S = \{\phi(x,f(x)): x \in U \}$.
Then we have a graphical smooth convergence of $P_t(S)$ to $U$ as 
as $t \rightarrow 1$.
\end{enumerate}

Property (3) is proved in 
\cite[Proposition 5.7]{mn4}. All other properties follow from the definition.

The importance of the above is that we can use the squeezing map to push 
a surface $S$ in a small tubular neighbourhood of $\Sigma$
towards $\Sigma$ while simultaneously decreasing its area.

In the rest of the paper we will say that a surface $S$ is graphical if it
satisfies
$S = \{\phi(x,f(x)): x \in U \}$ for some function $f$ and a 
subset $U \subset \Sigma$.

\subsection{The case of connected stable minimal surface.}

The following is a special case of our main interpolation result.  Setting $g=0$ in the statement of the proposition and assuming $\Gamma$ is connected, one can interpret it as a quantitative form of Alexander's Theorem.  Yet another way to interpret it is as a kind of Mean Curvature Flow performed ``by hand."

\begin{prop}[Interpolation / Quantitative Alexander Theorem] \label{main_deformation}
Let $\Si$ be a smooth connected orientable surface
of genus $g$,
with a map $P$ satisfying (1) and (2) above.
Let $\G \subset N_{h_0}(\Si)$ be a smooth embedded 
surface, such that each connected component of $\G$ has genus at most $g$.
For every $\delta>0$  there exists an isotopy $\G_t \subset  N_{h_0}(\Si)$ with 
\begin{enumerate}
\item $\G_0 = \G$
\item For each connected component $\G'$  of $\Gamma_1$
either $\G'= \phi (\Sigma \times t)$ for some $t \in [-1,1]$ 
or $\G'$ is contained in a ball 
of radius less than $\delta$
\item $Area(\G_t) \leq Area(\G)+\delta$ for all $t$. \label{eq: delta}
\end{enumerate}
\end{prop}

The reason that the $\delta$-constraint is important in (3) is that we will be gluing this isotopy into sweepouts with maximal area approaching the width $W$ and we want the maximal area of the resulting sweepout to still be $W$.

It follows from Alexander's theorem that any embedded two-sphere in $N_\varepsilon(\Sigma)\cong S^2\times [0,1]$ can be isotoped to either a round point or else to $\Sigma$ itself.  The difficulty is to obtain such an isotopy obeying the area constraint \eqref{eq: delta}.

It is instructive to consider the analogous question in $\mathbb{R}^3$ to that addressed in  Proposition \ref{main_deformation}.  Suppose one is given two embeddings $\Sigma_0$ and $\Sigma_1$ of two-spheres into $\mathbb{R}^3$.   We can ask whether for any $\delta>0$ there exists
an isotopy $\Sigma_t$ from $\Sigma_0$ and $\Sigma_1$ obeying the constraint (assuming $|\Sigma_1|>|\Sigma_0|$):
\begin{equation}
|\Sigma_t|\leq |\Sigma_1|+\delta \mbox{ for all } t.
\end{equation}

It is easy to see that the answer is ``yes."  Namely, one can even do better and find an isotopy satisfying 
\begin{equation}
|\Sigma_t|\leq |\Sigma_1| \mbox{ for all } t.
\end{equation}

To see this, one can first enclose $\Sigma_0$ and $\Sigma_1$ in a large ball about the origin $B_R$.   By Alexander's theorem there is an isotopy $\phi_t$ between $\Sigma_0$ and $\Sigma_1$ increasing area by a factor at most $A$ along the way.   First shrink $B_R$ into $B_{R/A}$, then perform the shrunken isotopy $(1/A)\phi_t$ on $B_{R/A}$, and then rescale back to unit size.

Of course, in $3$-manifolds that we must deal with in Proposition \ref{main} are $\mathbb{S}^2\times [0,1]$ in which one does not have good global radial isotopies to exploit.  However, the same idea of shrinking still applies if we first work locally in small balls to ``straighten" our surface.  We can also use the squeezing map to repeatedly press our surface closer to $\Sigma$ in the flat topology while only decreasing area.  

Let us explain the ideas in our proof of Proposition \ref{main_deformation} in more detail.
There are two main steps.  In the first, we introduce
a new local area-nonincreasing deformation process in balls.  The end result of applying this process in multiple balls centered around $\Sigma$ is to produce an isotopic surface $\Gamma$ consisting of $k$ parallel graphical sheets to $\Sigma$ joined by (potentially very nastily) nested, knotted and linked tubes.   

The local deformation we introduce exploits the fact that in balls, we can using Shrinking Isotopies to ``straighten" the surface while obeying the area constraint (a similar idea was used by Colding-De Lellis \cite{cd} in proving the regularity of $1/j$-minimizing sequences).   Our deformation process is a kind of discrete area minimizing procedure, somewhat akin to Birkhoff's curve shortening process.  In the process, it``opens up" any 
 folds or unknotted necks that are contained in a single ball. However,
 at this stage we can not open necks like on Figure \ref{fig:essential}.

After the first stage of the process, we are left with $k$ parallel graphical sheets arranged about $\Sigma$ joined by potentially very complicated necks.  If $k$ is $1$ or $0$, the proposition is proved. If not, the second step is to use a global deformation to deform the surface through sliding of necks to one in which two parallel sheets are joined by a neck contained in a single ball.  Then we go back to Step 1 to open these necks.  After iterating, eventually $k$ is  $1$ or $0$.

The second stage is complicated by the fact that the necks joining the various sheets can be nastily nested, knotted or linked.   We we need generalizations of the Light Bulb Theorem in topology to untangle this morass of cables and find a neck to open.
The version of the Light Bulb Theorem that will be most useful to us
is the following
(see Theorem \ref{general}).
Given a 3-manifold $M$ and two arcs, $\alpha$ and $\beta$, with boundary
points in $\partial M$ assume that one of the boundary points of $\alpha$
lies in the boundary component of $M$ diffeomorphic to a sphere. 
Then $\alpha$ and $\beta$ are isotopic as free boundary curves
if and only if they are homotopic as free boundary curves.
We will apply this theorem in the situation when $\alpha$
is a core arc of a ``cable", a collection of (partially) nested necks
in the tubular neighbourhood of $\alpha$.
In section \ref{sec:cables} we define cables and prove some auxiliary 
lemmas which allows us to treat these collections of tubes
almost as if it was an arc attached to the surface.

\begin{figure} 
   \centering	
	\includegraphics[scale=0.75]{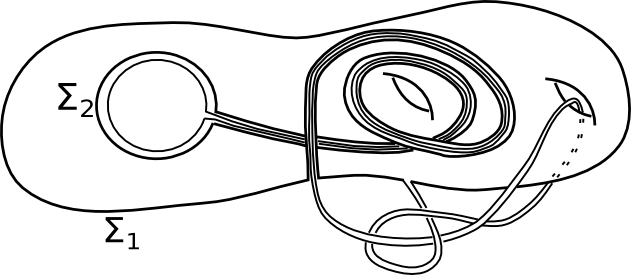}
	\caption{Surface $\G$ is within $\varepsilon$
	(in varifold norm) from $\Si_1 +2 \Si_2$, where $\Si_1$ is a stable
	minimal surface of genus 2 and $\Si_2$ is a stable minimal sphere. We can isotop
	$\G$ to $\Si_1$ while increasing its area by an arbitrarily small amount.}
	\label{fig:knotted}
\end{figure}

\subsection{The case of multiple connected components.}

Proposition \ref{main_deformation} deals with the situation
when surface $\G$ is contained in a tubular neighbourhood of
a connected 
stable minimal surface $\Si$. In general, we need to consider a situation when 
$\G$ clusters 
around a minimal surface $\Si$ that has multiple connected components.
This is illustrated on Figure \ref{fig:knotted}.
Surface $\G$ is mostly contained in the tubular neighbourhood
of minimal surfaces $\Si_1$ and $\Si_2$, while the part of
$\G$ outside of $N_h(\Si_1 \cup \Si_2)$ looks like a collection
of thin tubes that can link with each other and knot around
handles of $\Si_1$.
In this setting we prove
the following proposition.

\begin{prop}[Interpolation near disconnected stable minimal 
surface] \label{main_multiple}
Let $\Si_1$, ..., $\Si_k$ be pairwise disjoint embedded two-sided stable minimal surfaces in a $3$-manifold $M$ and denote $\Sigma:=\cup_i\Si_i$.
 
There exist $\varepsilon_0 >0$ and $h_0>0$ such that for all
$\varepsilon \in (0, \varepsilon_0)$ and $h \in (0,h_0) $ the following holds.
If $\G$ satisfies 

a) $Area (\G \setminus N_{h}(\Si))< \varepsilon$;

b) $genus(N_h(\Si_i) \cap \G) \leq genus(\Si_i)$ for each $i$;

then for every $\delta>0$  there exists an isotopy $\G_t$
with 
\begin{enumerate}
\item $\G_0 = \G$
\item $Area(\G_1 \setminus N_h(\Si)) < \delta$
\item $Area(\G_t) \leq Area(\G)+\delta$ for all $t$. 
\item $\Gamma_1$ is a surface consisting of some subcollection of the set $\Si$, joined by thin necks.
\end{enumerate}
\end{prop}

We describe ideas involved in the proof of Proposition \ref{main_multiple}.
First, we can use a version of Almgren's pull-tight flow together with maximum principle
for stationary varifolds to make the area of $\G$ outside of $N_{h}(\Si)$
arbitrarily small.
We are grateful to Andre Neves from whom we learned 
of arguments of this type. 

For each connected component $\Si$ we can intersect $\G$
with the tubular neighborhood $N_h(\Si)$ and glue in small discs 
to the boundary components of $N_h(\Si) \cap \G$ so as to obtain a closed
(possibly disconnected) surface.
Then we can apply Proposition \ref{main_deformation} to deform this surface into
disjoint graphical copies of $\Si_i$. Of course, we are not allowed to actually 
do any surgeries on $\G$. Instead, we perform deformations of 
Proposition \ref{main_deformation} while simultaneously moving thin necks attached to the surface
to preserve continuity.

After the surface has been deformed into a canonical form in the neighborhood
of each connected component $\Si_i$, we need a global argument, showing that one can always find 
a neck that can be unknotted, using Generalized Light Bulb Theorem, and slide
into the neighborhood of one of the $\Si_i$'s. This process terminates
only when for each $i$ surface $\G$ either avoids the neighbourhood of $\Si_i$
or looks like a single copy of $\Si_i$ with thin necks attached.

Note that unlike in the setting considered by Marques-Neves (Appendix A of \cite{mn4}), it is very important that we keep track of the part of the pulled-tight surface outside of the tubular neighborhood, as the neck we may ultimately need to find may pass through the complement of the tubular neighborhood.  See Figure \ref{fig:dumbbell} for an illustration of a case where this is necessary.  

\section{Deformation in the neighborhood of a connected
stable minimal surface}\label{onecomponent}

\subsection{Stacked surface.} \label{stacked_global}
Let $\G \subset N_{h_0}(\Si)$.
Given $\delta>0$ we will say that $\G$ is a $(\delta,k)$-stacked
surface
if there exists a decomposition
$\G = D \sqcup Y \sqcup X$ 
with the following properties:

a) $D= \bigsqcup_{i=1}^m D_i$, where each 
$D_i \subset \overline{D}_i $,
and $\{\overline{D}_i\}$ is a collection of disjoint
graphs over $\Si$, $\overline{D}_i= \{\phi(x,f_i(x))|x \in \Si \}$, satisfying 
$Area(\overline{D}_i ) < Area(\Si)
+ \frac{\delta}{10k}$;

b) $Area(\overline{D}_i \setminus D_i) < \delta/10$;

c) each connected component of $Y$
is a the boundary of a small tubular neighbourhood of 
an embedded graph, and their total area is at most $\delta/10$;

d) $X$ is a disjoint union of closed surfaces
of total area less than $\delta/10$,
each contained in a ball of radius less than
$\sqrt{\delta}/10$.

We can order the punctured surfaces $D_i$ to have descending height relative to a fixed unit normal on $\Gamma$, with $D_1$ the top-most.  Let us call $D_i$ the \emph{$i$th sheet}.  Let us call $Y$ the thin part of 
$\G$.

\begin{prop} \label{canoncial form}
Let $\Si$ be a strictly stable two-sided connected minimal surface
and $\G \subset N_{h_0}(\Si)$ be a (not necessarily connected)
smooth surface.
For any $\delta>0$ there exits $k>0$ and
an isotopy $\{ \G_t\}$
with $\G_0 = \G$, $\G_1$ is $(\delta,k)$-stacked
and $Area(\G_t) \leq Area(\G) + \delta$ for all $t$.
\end{prop}

\subsection{Choice of radius $r$ and open neighbourhood $\Omega_{h}$} \label{r}
Consider the projected current $P(\Gamma)$ supported on $\Sigma$.  There exists $r_1>0$ so that 
for any $r\leq r_1$ the mass of $P(\Gamma)$ in any ball $B_{r}(x)$ (with $x\in\Sigma$) is less than $\delta/200$.  
By continuity, we can choose $t_1$ close enough to $1$ so that 
for every $t \in [t_1,1)$ the mass of $P_t(\Gamma) \subset \Omega_{1 - t_1}$ in any ball $B_{r}(x)$ is at most $\delta/100$ for any $r<r_1$ and $x\in\Sigma$.   We replace $\Gamma$ with $P_{t_1}(\Gamma)$ (but do not relabel it).

Let $h \in (0, 1 - t_1)$, so that $\Gamma  \subset \Omega_h$.
Now we pick $r= r(\Si, \G, \delta)>0$, satisfying the following properties:


1) $r$ is smaller than the minimum of the convexity 
radii of $M$ and $\Si$;

2) $r < r_1$, that is, for every $x \in \Si$ and a ball $B_r(x)$ of radius $r$
we have that $Area(\G \cap B_r(x)) < \frac{1}{100} \delta$;

3) for every $x \in \Si$ and a ball $B_r(x)$ of radius $r$
we have that the exponential map $exp: B^{Eucl}_r(0) \rightarrow B_r(x)$
satisfies $0.99 < |d exp_y| < 1.01$ for all $y \in B_r(x)$.

\subsection{Choice of triangulation and constant $c$.} \label{triangulation}
Fix a triangulation 
of $\Si$,
so that for each 2-simplex $S_i$, $1 \leq i \leq m$, in the triangulation
there exists a point $p_i \in S_i$ with 
$S_i \subset B_{r/3}(p_i)$.
Assume that the number is chosen so that
$S_{i+1}$ and $S_i$ share an edge.
We cover $\Omega_{h}(\Si)$ by a collection of
cells $\{\Delta_i = \phi (S_i \times [-(1-t_1),(1-t_1)]) \}$.
The interiors of $\Delta_i$'s are disjoint and
each $\Delta_i$ is contained in a ball  $B_{r/2}(p_i)$.
Let $C_i =  \phi ( \partial S_i \times [-h,h]) \subset \partial \Delta_i$.
Let $c = \min\{ Area(\Si \cap \Delta_i)\}$.

By applying squeezing map $P_t$ we may assume that
$\Omega_h \subset N_{r/10}(\Sigma)$ and $\G$ is contained in the union of $\Delta_i$.

We will first need to prove a local version of
Lemma \ref{canoncial form}. Namely,
we will show that $\G$ can be deformed
into certain canonical form in each cell $\Delta_i$.

We introduce several definitions.

\subsection{Essential multiplicity.}
Suppose $r' \in (r/2, r)$, $\G \subset \Omega_h(\Si)$ and
assume $\G$ intersects $\partial B_{r/2}(p_i)$
and $\partial B_{r}(p_i)$ transversally.
Let $\mathcal{S}(\G,p_i, r')$ denote the set of surfaces 
$S \subset \Omega_{h}(\Si)$, 
such that $S$ intersects $\partial B_{r/2}(p_i)$ transversally
and 
there exists an isotopy from $S$ to $\G$
through surfaces $S'$
such that $S' \setminus int(B_{r'}(p_i)) = \G \setminus int(B_{r'}(p_i))$.
Let $k(S)$ denote the number of connected components of
$S \cap \partial B_{r/2}(p_i)$, which are not contractible
in $\partial B_{r/2}(p_i) \cap \Omega_{h}(\Si)$.
We define the essential multiplicity of $\G$ in
$B_{r/2}(p_i)$ to be 
$k_{ess}(\G,i,r') = \inf \{k(S)| S \in \mathcal{S}(\G,p_i,r') \}$.

We have the following lemma.

\begin{lemma} \label{essential}
$Area(\G \cap B_{r'}(p_i)) \geq k_{ess}(\G,p_i,r') Area(\Si \cap B_{r'}(p_i)) - O(h)$.
\end{lemma}

\begin{proof}
By coarea inequality $Area(\G \cap B_{r'}(p_i)) \geq \int_{\rho=0}^{r'} L(\partial B_\rho(p_i) \cap \G) d \rho$.
For almost every $\rho$ we have that the number of 
connected components of
$\G \cap \partial B_{\rho}(p_i)$, which are not contractible
in $\partial B_{\rho}(p_i) \cap \Omega_{h}(\Si)$ is at least $k_{ess}(\G,i,r')$.
Indeed, otherwise we could radially isotop $\G$ to obtain a surface 
with fewer non-contractible components of $\G \cap \partial B_{r/2}(p_i)$,
contradicting the definition of $k_{ess}(\G,i,r')$.

We have $L(\partial B_\rho(p_i) \cap \G) \geq L(\partial B_\rho(p_i) \cap \Si) - O(h)$ and the lemma follows.
\end{proof}

Note that it may happen that the
relative map 
$$(\G \cap B_{r'}(p_i), \G \cap \partial B_{r'}(p_i)) \rightarrow (B_{r'}(p_i),\partial B_{r'}(p_i))$$
is null-homotopic (but not null-isotopic) and yet $k_{ess}(\G,i) \neq 0$
(see Fig. \ref{fig:essential}).

\begin{figure} 
   \centering	
	\includegraphics[scale=0.75]{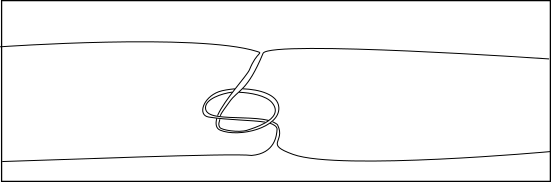}
	\caption{Two graphical sheets joined by a knotted neck.
	There is a homotopy, but there is no isotopy
	pushing the surface into the boundary of the cell, so $k_{ess} = 2$.}
	\label{fig:essential}
\end{figure}

\subsection{Surfaces stacked in a cell.}
An embedded surface $S$ is
$(\delta,k)$-stacked in a cell $\Delta_i$ if there exists
a decomposition 
\begin{equation} \label{division}
S \cap \Delta_i = D \sqcup Y \sqcup X
\end{equation}
with the following properties:

a) $D= \bigsqcup_{i=1}^m D_i$, where each 
$D_i \subset \overline{D}_i $,
and $\{\overline{D}_i\}$ is a collection of disjoint
graphs over $S_i$, $\overline{D}_i= \{\phi(x,f_i(x))|x \in S_i \}$, satisfying 
$Area(\overline{D}_i ) < Area(S_i)
+ \frac{\delta}{10k}$;

b) $Area(\overline{D}_i \setminus D_i) < \delta/10$;

c) each connected component of $Y$
is a boundary of a small tubular neighbourhood of 
an embedded graph, and their total area is at most $\delta/10$;

d) $X$ is a disjoint union of closed surfaces
of total area less than $\delta/10$,
each contained in a ball of radius less than
$\sqrt{\delta}/10$.

We can order the punctured surfaces $D_i$ to have descending height relative to a fixed unit normal on $\Gamma$, with $D_1$ the top-most.  Let us call $D_i$ the \emph{$i$th sheet}.  Let us call $Y$ the thin
part of $S \cap \Delta_i$.




\subsection{Key lemmas used in the proof of Proposition \ref{canoncial form}}.

%

The following is the blow down - blow up lemma
from \cite{cd}.

\begin{lemma} \label{radial}
Suppose $B_{r'}(x)$ is a ball of radius
$r' \leq r$ and ${\G_t}$ 
be an isotopy with $\G_t \setminus B_{r'}(x)=
\G_0 \setminus B_{r'}(x)$.
Then there exists an isotopy
${\overline{\G}_t}$, such that:

(a) $\overline{\G}_0=\G_0$ and $\overline{\G}_1=\G_1$;

(b) $\overline{\G}_t \setminus B_{r'}(x) =\G_0 \setminus B_{r'}(x)$;

(c) $Area(\overline{\G}_t) \leq \max\{Area(\G_0), Area(\G_1) \}+2 r' L(\G_0 \cap 
\partial B_{r'}(x))$ for $t \in [0,1]$.

\end{lemma}

\begin{proof}
For the proof see radial deformation construction 
in \cite{cd}, Step 2 in the proof of Lemma 7.6.
\end{proof}

\begin{lemma} \label{stacking}
Let $\G$ and $\Delta_i$ be as defined above.
There exists an admissible family $\{\G_t\}$
and $k \leq k_{ess}(\G,i,3r/4)$,
such that:

a) $\G_0 = \G$ and $\G_t \setminus int(B_{r}(p_i)) = \G \setminus int(B_{r}(p_i))$
for all $t$;

b) $\G_1$ is $(\delta,k)$-stacked in $\Delta_i$ and
$Area(\G_1)< k Area(\Si \cap \Delta_i) + \delta/2$;

c) $Area(\G_t)< Area(\G) + \delta$ for $t \in [0,1]$;

d) if $\G$ is $(\delta',k')-$stacked in a cell $\Delta_{j}$,
$j = i-1$ or $i+1$, 
then either $k=k'$ and $\G_1$ is 
 $(\delta',k)-$stacked in $\Delta_j$
 or $k_{ess}(\G_1,j, 3r/4)< k'$.
\end{lemma}

\begin{proof}

After applying the squeezing map $P$
we may assume that $\G$ is contained in $\Omega_h(\Si)$,
where $h$ sufficiently small, so that:

 - $Area( \phi(\Si,h') \cap B_r(p_i)) \leq Area(B_{r}(p_i) \cap \Si) + \frac{\delta}{20k}$ for all $h' \in [-h,h]$;
 
 - $Area(\partial B_{r'}(p) \cap \Omega_h(\Si)) < \delta r' /100$ for
 all $r' \in (0,r)$.

By coarea inequality and the definition
of $r$ (\ref{r}) there exists
a radius $r' \in [3r/4, r]$
with 

\begin{equation} \label{eqn:coarea}
 L( \G \cap \partial B_{r'}(p_i)) \leq \frac{\delta}{10 r}
\end{equation}

Let $k = k_{ess}(\G, p_i, r')$.
It follows from the definition that $k \leq k_{ess}(\G, p_i, 3r/4)$.

We will show that there exists an 
isotopy $\{ \G_t\}$ of $\G$ that does not change $\G$
outside of the interior of $B_{r'}(p_i)$ and deforms
it to a surface $\G_1$ with the following properties:

(a) $\G_1$ is $(\delta,k)$-stacked in $\Delta_i$

(b) $Area(\G_1 \cap B_{r'}(p_i)) \leq Area(\G_0 \cap B_{r'}(p_i)) + \delta/10$.

Then by Lemma \ref{radial} and (\ref{eqn:coarea}) we may assume that 
the isotopy $\{ \G_t \}$ also satisfies
$Area(\G_t) \leq Area(\G) + \delta$ for $t\in [0,1]$.
In other words, in the construction below
we do not need to control the areas of the intermediate 
surfaces.

We start by deforming all connected components of $\G$
which are closed surfaces in the interior of $B_{r'}(p_i)$,
so that they lie in a small ball and have total area
less than $\delta/100$. From now on, without any loss of
generality, we may assume that every connected component of 
$\G \cap B_{r'}(p_i) $ intersects $\partial B_{r'}(p_i)$.

By definition of $k_{ess}(\G,i, r')$ we can deform
$\G$ into a surface $S$, such that  exaclty
$k$ connected components of 
$S \cap \partial B_{r/2}(p_i)$ are not contractible
in $\partial B_{r/2}(p_i) \cap \Omega_{h}(\Si)$.

Choose a collection of embedded mutually disjoint closed 
curves $\{\gamma_j \}$, $\gamma_j \subset S$,
such that  connected components of $S \setminus \cup \gamma_j$
are discs, annuli or pairs of pants.
Moreover, collection of curves $\{\gamma_j \}$ can be chosen so 
that it includes all connected components of
$S \cap \partial B_{r/2}(p_i)$.
We will say that a curve $\gamma \subset S \cap \partial B_{r/2}(p_i)$
(resp.  $\gamma \subset S \cap \partial B_{r'}(p_i)$) is essential
if it is non-contractible in $\partial B_{r/2}(p_i) \cap \Omega_{h}(\Si)$
(resp. $\partial B_{r'}(p_i) \cap \Omega_{h}(\Si)$).

We may assume that $S$ has been deformed in such a way that
 
1) every essential $\gamma \subset S \cap \partial B_{r/2}(p_i)$
 is a latitudinal circle, that is $\gamma = \partial B_{r/2}(p_i) 
 \cap \phi (\Si \times t)$ for some $t \in (-h,h)$;
 
2) every non-essential $\gamma \subset S \cap \partial B_{r/2}(p_i)$
is of the form $\gamma = \partial B_{\rho(\gamma)}(x(\gamma)) \cap \partial B_{r/2}(p_i)$
with the total sum of the areas of all 
 $B_{\rho(\gamma)}(x(\gamma)) \cap \partial B_{r/2}(p_i)$ less than
 $\delta/100$.
 
Let $S'$ be a connected component of $S \setminus \cup \gamma_j$
that lies in $B_{r/2}(p_i)$.
If $S'$ is a disc with a non-essential boundary 
in $ \partial B_{r/2}(p_i)$ we can isotop
it to a small cap near its boundary
and push it out of $B_{r/2}(p_i)$.
Similarly, if $S'$ is an annulus
or a pair of pants with non-essential boundary components
we can isotop it to a surface given by the boundary of a tubular neighbourhood
of a curve or a Y graph with the area at most 
$2 \sum_{l} Area(B_{\rho(\gamma_{j_l})}(x(\gamma_{j_l})) \cap \partial B_{r/2}(p_i))$,
where $\gamma_{j_l}$ are boundary components of $S'$.

If $S'$ is a disc with an essential
boundary curve we isotop it
to $B_{r/2}(p_i) 
 \cap \phi (\Si \times t)$.
Similarly, we isotop an annulus or a pair of pants
with $m=1,2$ or $3$ essential boundary components
to a surface given by $m$ stacked discs with holes
connected by narrow tubes or boundaries
of a tubular neighborhood of a graph.

Ambient isotopy theorem guarantees
that these deformations can be done
so that different connected components 
do not intersect each other.
As a result we obtain that the new surface $\G_1$
is $(k, \delta)$-stacked
in $\Delta_i \subset B_{r/2}(p_i)$.

We would like to deform connected components of
$S \setminus \cup \gamma_j$ that lie in 
$B_{r'}(p_i) \setminus B_{r/2}(p_i)$
in a way that will guarantee the 
upper bound on the area and property 
d) in the statement of the Lemma.

The main issue is that our deformation
is not allowed to change the boundary
$S \cap \partial B_{r'}(p_i)$, which can be very wiggly.
However, for some sufficiently small positive $\delta'< r' -r/2$ 
we can deform $S$ so that
$S \cap \partial B_{r'-\delta'}(p_i)$
satisfies the same properties 1)-2) as
$S \cap \partial B_{r/2}(p_i)$,
while controlling the area of
$S \cap (B_{r'}(p_i) \setminus B_{r'-\delta'}(p_i))$
in terms of $h$.
We choose $\delta'$ sufficiently small, so
that the distance function to $p_i$ restricted
to $S \cap (B_{r'}(p_i) \setminus B_{r'-\delta'}(p_i))$
is non-degenerate.

First, we deform the collars of non-essential curves
$\gamma \subset S \cap \partial B_{r'}(p_i)$,
so that their intersection with $\partial B_{r'-\delta'}(p_i)$
satisfies condition analogous to 2) above.
This can be done in a way so that the area of the deformed part of 
$S \cap (B_{r'}(p_i) \setminus B_{r'-\delta'}(p_i))$
is bounded by the area of the disc $\gamma$
bounds in $\partial B_{r'}(p_i)$.

Now we would like to straighten the essential curves.
Let $\gamma$ denote the highest (with respect
to signed distance from $\Si$) essential curve
in $ S \cap \partial B_{r'}(p_i)$. Choose $t \leq h$,
so that the latitudinal curve $\phi (\Si \times t) \cap \partial B_{r'-\delta'}(p_i)$
lies above $\gamma$. We isotop the small non-essential necks
in the neighbourhood of $\partial B_{r'-\delta'}(p_i)$ so that their
intersection with $\partial B_{r'-\delta'}(p_i)$
lies either above $ \phi(\Si \times t)$ or below
$\gamma$. After this deformation the subset of 
$ \partial B_{r'-\delta'}(p_i) \setminus S$ that lies
between $\phi(\Si \times t) \cap \partial B_{r'-\delta'}(p_i)$ 
and $\gamma$ is homeomorphic to
a cylinder. This implies that there exists an isotopy of $S$
sliding the essential intersection $\gamma$ to 
$\phi(\Si \times t) \cap \partial B_{r'-\delta'}(p_i)$. 

We iterate this procedure for every essential curve
in $S \cap \partial B_{r'-\delta'}(p_i)$
deforming them into latitudinal curves.
The isotopies done in this way have the property that
the area of the deformed part of 
$S \cap (B_{r'}(p_i) \setminus B_{r'-\delta'}(p_i))$
is bounded by the area of $\partial B_{r'-\delta'}(p_i) \cap \Omega_{h}(\Si)$.

We conclude that total area of
$S \cap (B_{r'}(p_i) \setminus B_{r'-\delta'}(p_i))$
after the deformation
goes to $0$ as $h \rightarrow 0$.

Suppose the
collection of curves $\{\gamma_j \}$ is chosen so 
that it includes all connected components of
$S \cap \partial B_{r' - \delta'}(p_i)$.
Suppose $S'$ is a connected component of
$S \setminus \cup \gamma_j$ that lies in 
$( B_{r'-\delta'}(p_i) \setminus B_{r/2}(p_i))$.

If all boundary components of $S'$
are non-essential, we can deform it 
so that it is a boundary of a tubular neighbourhood
of a curve or a Y graph.

If $S'$ is an annulus and one of its boundary
components is essential then the second boundary component
must also be essential
(this follows by examining the homomorphism of fundamental 
groups induced by inclusion).
Observe that if both boundary components lie in 
$\partial B_{r/2}(p_i)$ we obtain a contradiction
with the definition of $k_{ess}$.
If both lie in $\partial B_{r'-\delta'}(p_i)$
we push $S'$ very close to $\partial B_{r'-\delta'}(p_i)$,
so that its area is at most $Area(\partial B_{r'-\delta'}(p_i) 
\cap \Omega_{h}(\Si))$.
If one component of $\partial S'$ lies in $\partial B_{r'-\delta'}(p_i)$
and another component lies in $\partial B_{r/2}(p_i)$
we can isotop $S'$ so that it is a graphical sheet
of area at most $Area(\Si \cap (B_{r'-\delta'}(p_i) \setminus B_{r/2}(p_i))) + O(h)$.

Suppose now that $S'$ a pair of pants.
It follows by examining the homomorphism
from $\pi_1(S') = \mathbb{Z} * \mathbb{Z}$ to $\pi_1((B_{r'-\delta'}(p_i) \setminus B_{r/2}(p_i)) \cap \Omega_{h}(\Si)) = \mathbb{Z}_1$
that $S$ can have either $0$ or $2$ essential boundary components.
In both cases we can deform it similarly to 
the case of an annulus, but with a narrow tube attached.

By Lemma \ref{essential} we have that
the area bound c) is satisfied for $h$
sufficiently small.

It is straightforward to check that the above
deformations can be done so that if
$\Delta_j \cap B_{r'}(p_i) \neq \emptyset$
and $\G$ was $(k',\delta')$-stacked in $\Delta_j$
for $k' = k$, then it will be 
$(k,\delta)$-stacked after the deformations.

Suppose $\G$ was $(k',\delta')$-stacked in $\Delta_j$
for $k' > k$. Then after the deformation
there will be some open subset $U \subset S_j$,
such that for every $x \in S_j$ we have
$P^{-1}(x) \cap \G_1 $ has less than $k$ points.
It follows that $k_{ess}(\G_1,j, 3/4)< k'$.
\end{proof}

\subsection{Proof of Proposition \ref{canoncial form}.}

Fix $\delta>0$. 
First we construct a deformation of $\G$
to a surface that is $(\delta/10,k)$-stacked in each
cell $\Delta_i$ for some integer $k$, while increasing its area by at most $\delta/2$.

Recall the definition of $c$ from (\ref{triangulation}).
Let $\delta_i = \min\{c/2, \frac{1}{2^i} \frac{\delta}{100}\}$.
We will construct a 
sequence of surfaces $\G^0,...,\G^N$, such that

1*) $\G^1= \G$ and 
$\G^N$ is $(\delta/10,k)$-stacked in each
cell $\Delta_i$.

2*) $Area (\G^{i+1}) \leq Area(\G^{i})+ \delta_i$
and for every $p_j$,  
$$Area (\G^{i+1} \cap B_r(p_j)) \leq Area(\G^{i} \cap B_r(p_j))+ \delta_i$$

3*) There exists an isotopy $\{\G_t^i\}$
with $\G_t^i \subset N_{h_0}(\Si)$,
such that $\G_0^i=\G^i$, $\G_1^i=\G^{i+1}$ and 
$Area(\G_t^i) < Area(\G^i) + \delta/2$ for $t \in [0,1]$.


The process consists of a finite number of iterations.
The $l$'th iteration will consist of $m_l \leq m$ steps.
Let $\tilde{m}(l) = \sum_{l' \leq l} m_{l'}$.
For $j=1,..,m_l$ we  deform 
$\G^{\tilde{m}(l-1)+j-1}$ into $\G^{\tilde{m}(l-1)+j}$.
At the $j$'th step of $l$'th iteration we
apply Lemma \ref{stacking} to $\G^{\tilde{m}(l-1)+j-1}$
to construct an isotopy to $\G^{\tilde{m}(l-1)+j}$,
which is $(k_{\tilde{m}(l-1)+j}, \delta_{\tilde{m}(l-1)+j})$-stacked
in the cell $\Delta_j$. 
Now by induction and Lemma \ref{stacking} d) we have two possibilities:

1) $\G^{\tilde{m}(l-1)+j}$ is
$(\delta/10, k_{\tilde{m}(l-1)+j})$ stacked in 
cells $\Delta_1, ..., \Delta_j$; 

2) $\G^{\tilde{m}(l-1)+j}$ is $(k', \delta/10)$-stacked
in $\Delta_{j-1}$ and
$k_{ess}(\G^{\tilde{m}(l-1)+j}, j-1, 3r/4)< k'$.

In the second case we apply Lemma 
\ref{stacking} to $\G^{\tilde{m}(l-1)+j}$ in 
the cell $\Delta_{j-1}$. This deformation (preceded
by an application of a squeezing map $P$ if necessary) will,
by Lemma \ref{essential},
reduce the area of the surface by at least $c - \delta_{\tilde{m}(l-1)+j}> c/2$.
Since the area of $\G_n$ can not be negative,
we must have that eventually it is stacked
in every cell. The total area increase after all the deformations
is at most $\sum \delta_n < \delta/10$.
This concludes the proof of 
Proposition \ref{canoncial form}.

\section{Tubes, cables and root sliding} \label{sec:cables}

\begin{defi} \label{tube}
(Definition of a tube.) Let $\gamma:[0,1] \rightarrow N$ be an embedded 
curve and $\exp_{\gamma}: [0,1] \times D^2$
be the normal exponential map and suppose 
$\exp_{\gamma}$ is a diffeomorphism onto
its image for $v \in D^2$ with $|v| \leq 2 \varepsilon$.
We will say that $T = \{\exp_{\gamma}(t,v): |v| = \varepsilon \}$
is an $\varepsilon$-tube with core curve $\gamma$.
\end{defi}

In this paper we will often need to isotopically deform
parts of a surface so that it looks like a disjoint union of long
tubes. We will then need to move these tubes around in 
a controlled way.
Here we collect several definitions and lemmas
related to this procedure.

\begin{defi} \label{def:cable}
Let $\G$ be an embedded surface in $M$.
$\mathcal{C} = \{ (A_i, \gamma_i, \varepsilon_i)\}_{i=1}^k$ will be called a \textbf{cable}
of thickness $\varepsilon>0$ with \textbf{root balls} $B_1$ and $B_2$
and necks $A = \bigcup_{i=1}^k A_i$ ,
where

(1) $\{ A_i \subset \G \}$ is a collection of disjoint 
$\varepsilon_i$-tubes, $\varepsilon_i \leq \varepsilon$ with core curves $\gamma_i$;

(2) $B_1$ and $B_2$ are disjoint opens 
balls of radius $r> \varepsilon$.
For $j=1,2$ we have that $B_j \cap \G = \sqcup_{i=1}^k {\tilde A}_{j,i}$,
where each ${\tilde A}_{j,i}$ is homeomorphic to an annulus with boundary
circles ${\tilde c}_{j,i}^1$ and ${\tilde c}_{j,i}^2$, 
satisfying ${\tilde c}_{j,i}^1 \subset A_i$ and 
 ${\tilde c}_{j,i}^2 \subset M \setminus A$ and 
 with $\gamma_i(0) \subset B_1$ and $\gamma_i(1) \subset B_2$;


(3) let $N_i$ denote the solid cylinder bounded by $A_i$, 
$N_i=\{\exp_{\gamma_i}(t,v): |v| \leq \varepsilon_i \}$,
then $N_i \subset N_1$ for all $i$.
\end{defi}

In the following Lemma \ref{neck squeeze} we observe that if a cable has 
sufficiently small thickness then we can squeeze it
towards the core curve $\gamma_1$ to make the total
area of necks arbitrarily small.

\begin{lemma} \label{neck squeeze}
There exists a constant $C_{sq}>0$,
such that for all sufficiently small $\varepsilon>0$
the following holds. If 
$\G_0$ is a surface with 
a cable $\mathcal{C} = \{ (A_i, \gamma_i, \varepsilon_i)\}_{i=1}^k$ 
of thickness $\varepsilon>0$
with root balls $B_r(p_1) $ and $B_r(p_2)$ then 
there exists an isotopy $\G_t$, $t \in [0,1)$ such that:

(1) $Area(\G_t) \leq Area(\G_0) + C_{sq} k \varepsilon^2$ for all
$t \in [0,1)$;

(2) $\G_t$ is a surface with 
a cable $\mathcal{C}_t = \{ (A_i^t, \gamma_i^t, \varepsilon_i^t)\}_{i=1}^k$ 
of thickness $\varepsilon_t = (1-t)\varepsilon$
with necks $A_t$
and root balls $B_{r_t}(p_1)$ and $B_{r_t}(p_2)$ of radius $r_t$;

(3) $\varepsilon_i^t$, $r_t$ and  $Area(A_t)$ 
are monotone decreasing functions of $t$ with
$A_t \rightarrow 0$, $\varepsilon_i^t \rightarrow 0$ 
and $r_t \rightarrow 0$ as $t \rightarrow 1$;

(4) $\gamma_1^t = \gamma_1$ for all $t \in [0,1)$.

\end{lemma}

Before proving Lemma \ref{neck squeeze}
we state the following auxiliary result.

\begin{lemma} \label{patching in the ball}
Let $B_1(0)$ be a ball in $\R^3$ and $\G \subset B_1(0)$
be a surface with $\partial \G \subset \partial B_1(0)$.
Let $\gamma_t$ be an isotopy of curves in $\partial B_1(0)$
with $\gamma_0 = \partial \G$ and $l(\gamma_t) < L$. Then
there exists an isotopy
$\G_t$ with $\G_0 = \G$, $\partial \G_t = \gamma_t$ and
$Area(\G_t) < Area(\G_0) + L$.
\end{lemma}

\begin{proof}
The result follows by the blow down - blow up trick from \cite{cd} as in 
the other parts of this paper.
\end{proof}

\begin{proof}
Let $\phi_t^1: N_1 \rightarrow N_1$ be a map given by
$\phi_t^1(exp_1(v)) = exp_1 (tv)$ for $v \in N \gamma_1$.

Choose monotone decreasing functions 
$f_i:[0,1] \rightarrow [0,1]$,
$2 \leq i \leq k$, so that the map
$\phi_t^i: N_i \rightarrow M$ defined by
$\phi_t^i(v) = \phi_t (exp_1^{-1} (exp_i(f_i(t) v )))$
is a diffeomorphism onto its image and 
$A_i^t = \phi_t^i(A_i)$ are all disjoint.

This defines the desired isotopy outside
of the root balls $B_r(p_1)$ and $B_r(p_2)$.
We extend the istopy inside the balls
using Lemma \ref{patching in the ball}.
This finishes the proof of Lemma \ref{neck squeeze}.
\end{proof}

Given a surface $\G$ with a cable we define 
a new surface obtained by sliding the root ball $B_1$
 as illustrated on Fig. \ref{fig:sliding}.

\begin{defi} \label{def:cable surgery}
Let $\G_0$ be a surface  with cable
$\mathcal{C} = \{ (A_i, \gamma_i, \varepsilon_i)\}_{i=1}^k$ 
of thickness $\varepsilon>0$ with root balls $B_1$ and $B_2$.
Let ${\tilde c}_{j,i}^1$ and ${\tilde c}_{j,i}^2$
be as in Definition \ref{def:cable}.
Let $B_3$ be a ball intersecting $\G_0$ in a disc and
$\alpha \subset (\G_0 \setminus (B_1 \cup B_2 \cup B_3)$ an arc with endpoints
$\alpha(0) \subset \tilde{c}^2_{1,1}$ and $\alpha(1) \subset \partial B_3$.

Let $\beta$ denote an arc in $\partial B_1$
connecting the endpoint of $\alpha(0)$ to the endpoint 
of $\gamma_1 \setminus B_1$. Let $\tilde{\gamma}_a = \alpha \cup \beta
\cup (\gamma_1 \setminus B_1)$.
Perturb $\tilde{\gamma}_a$ in the direction normal
to $\G_0$, so that it does not
intersect $\G_0 \setminus \cup A$ except at the endpoints. 
Let $\tilde{\gamma}_b$ denote an arc obtained by perturbing
$\alpha$ to the other side of $\G_0$ from $\tilde{\gamma}_a$.

We will say that $\G_1$ is obtained from $\G_0$
by $\varepsilon$-cable sliding if the following holds:

(a) $\G_1$ has a cable 
$\mathcal{C}_1 =  \{ (A_i^1, \gamma_i^1, \varepsilon_i^1)\}_{i=1}^k$
with $\varepsilon_i^1 \leq \varepsilon_1$, $\gamma_1^1 = \tilde{\gamma}_a$
and root balls $B_3$ and $B_2$;

(b) $\G_1$ has a cable 
$\mathcal{C}_1 =  \{ (A_i^2, \gamma_i^2, \varepsilon_i^2)\}_{i=2}^{k}$
with $\varepsilon_i^1 \leq \varepsilon_1$, $\gamma_2^2 = \tilde{\gamma}_b$
and root balls $B_1$ and $B_3$;

(c) Setting $A^1 = \cup_{1}^k A_i^1$
and $A^2 = \cup_{2}^k A_i^2$ we have 
$\G_1 \setminus (A_1 \cup A^2 \cup B_1 \cup B_3) 
= \G_0 \setminus (A \cup B_1 \cup B_3)$.
\end{defi}

The following lemma allows us to slide the root
of a cable along a curve contained in $\G$ 
while increasing its area by in a controlled way.

\begin{figure} 
   \centering	
	\includegraphics[scale=0.85]{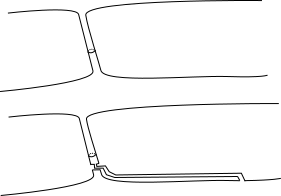}
	\caption{Sliding the root of the neck.}
	\label{fig:sliding}
\end{figure}

\begin{lemma} \label{lem: sliding}
For every $\varepsilon_0>0$ there exists $\varepsilon>0$ sufficiently small,
so that if $\G_1$ is obtained from $\G_0$
by $\varepsilon$-cable sliding then
there exists an isotopy $\G_t$, $t \in [0,1]$, such that
$Area(\G_t) \leq Area(\G) + \varepsilon_0$ for all $t \in [0,1]$.
\end{lemma}

\begin{proof}
By Lemma \ref{neck squeeze} we may assume that the thickness
$\varepsilon$ and the radius of the root balls $r$ are as small as we like.

Let $I_L = \{ (x,0,0) | 0 \leq x \leq L \} \subset \R^3$
and
$S_{xy}$ denote the $xy$-plane in $\R^3$.
Let $L= l(\alpha)$.
For every $c>0$ there
exists $\varepsilon_1>0$
and a diffeomorphism
$\Phi: N_{\varepsilon_1}(\alpha) \rightarrow N_{\varepsilon_1}(I_L) \subset \R^3$,
such that


(a) $\Phi(\alpha) = I_L$;

(b) $\Phi(S \cap N_{\varepsilon_1})= S_{xy} \cap N_{\varepsilon}(I_L)$;

(c) $1-c\leq  ||D \Phi|| \leq 1+c$.

\noindent
Fix $c<1/10$ to be chosen later (depending on $\varepsilon_0$) and 
assume $20 \varepsilon < \varepsilon_1$ and $2r < \varepsilon_1$.
Let $q_a= = \Phi(\alpha(0))$ and $q_b = \Phi(\alpha(1))$.
By our choice of $c$ we have $ \Phi(B_{1}) \subset B_{\varepsilon_1/2}(q_a)$

It is straightforward to construct a 1-parameter family 
of diffeomorphisms $\phi_t: N_{\varepsilon}(I_L) \rightarrow N_{\varepsilon}(I_L)$,
$\phi_0 = id$, and
generated by a 1-parameter family of compactly supported vector fields
 $\xi_t$ with the following properties:
 
 (i) $\phi_t(B_{\varepsilon_1/2}(q_a))$ is an isometric copy of 
 $B_{\varepsilon_1/2}(q_a)$ translated distance $tL$ along the
 $x$ axis;
 
 (ii) $\xi_t(p)$ lies in $S_{xy}$ for every $p \in S_{xy}$;
 
 (iii)  $\xi_t$ is supported in $N_{2 \varepsilon_1/3}(\alpha)$ for all $t$;
 
 (iv) $l(\phi_t(\Phi(\gamma_i)) )< 10 l(\gamma_i)$ for all $t$.

Composing with $\Phi^{-1}$ we obtain a $1$-parameter 
family of diffeomorphisms $\tilde{\phi}_t: M \rightarrow M$.
Observe that by condition (ii) the restriction of $\tilde{\phi}_t$ to
$S$ is a diffeomorphism of $S$, in particular, 
$Area (\tilde{\phi}_t(S)) = Area(S)$.
By compactness we can choose $\varepsilon \in (0, \varepsilon_1/20)$ 
sufficiently small, so that $Area(\tilde{\phi}_t(S)) \leq \frac{1}{10 k} \varepsilon_0$
for $t \in [0,1]$. 
Moreover, we can isotop each surface $\tilde{\phi}_1(A_i)$ so that 
it coincides with $\partial N_{\varepsilon_i}(\tilde{\phi}_1(\gamma_i))$.

This finishes the construction of the desired isotopy.
\end{proof}

\begin{lemma} \label{lem: tree to necks}
For every $\varepsilon>0$ there
exists a $\delta>0$ with the following property. 
Suppose a surface $\G$ is in a canonical form, in particular it is $(\delta,m)$-stacked in
every ball $B_r(x_j)$. There exists an isotopy of 
$\G$, increasing the area of $\G$ by at most $\varepsilon$,
so that the thin part $T = \sqcup A_i$, where each $A_i$
is homeomorphic to an annulus.
\end{lemma}

\begin{proof}
If every connected component of $T$ has
2 boundary components then we are done.

Suppose $T_1$ is a connected component with $k \geq 3$
boundary components.
We describe how to use the root sliding lemma
to deform $T_1$ into two disjoint thin
subsets, each having a smaller number of boundary
components.

Since the surface is in a canonical form,
there exists a cell $\Delta_j$ with $\gamma = \partial \Delta_j \cap T_1$
non-empty.
Let $\gamma_1$ denote an inner most
closed curve of $\gamma$.
Let $D_1$ denote the small disc $\gamma_1$ bounds in  $\partial \Delta_j$.
We have that interior of $D_1$ does not intersect 
$T_1$.
(Note, however, that there could be
 connected components of 
 $(T \setminus T_1) \cap \partial \Delta_j$
intersecting $D_1$).

We consider two possibilities:

1) $\gamma_1$ bounds a disc $D \subset T_1$.
Then we can find a ball $B\subset N_{h_0}(\Si)$
with $\partial B = D_1 \cup D$.
Let $\tilde{B}$ denote a small ball 
with $D_1 \subset \partial \tilde{B}$ 
and $int(\tilde{B}) \cap int(B) = \emptyset$.
There exists a diffeotopy $\Phi_t$
of $N_{h_0}(\Si)$, such that 
$\Phi_1(B) \subset  \tilde{B}$.
It is straightforward to check 
that using repeated application 
of the blow down - blow up trick Lemma \ref{radial}
we can make sure that the areas of $\Phi_t(\G)$
do not increase by more than $O(\delta)$.
In the end, we obtain that the number of 
connected components of $T_1 \cap \partial \Delta_j$
has decreased by one.

2) $\gamma_1$ separates  connected components
of $\partial T_1$. Let $A$ denote the 
component of $T_1 \setminus \gamma_1$,
which has more than 2 boundary components.
Let $\alpha$ be a path in $A$ from $\gamma_1$
to a different boundary component of $A$
and into the thick part of $\G$.
Let $\tilde{\gamma}$ denote all connected 
components of $T$ that are contained inside
a small disc bounded by $\gamma_1$ (including
$\gamma_1$).
Let $\tilde{T}$ denote a small neighbourhood
of $\tilde{\gamma}$ in $T$.
In a small neighbourhood of $D_i$
we can isotop $\tilde{T}$ so that it satisfies 
the properties of a cable with two roots.
We can then use Lemma \ref{lem: sliding}
to move one of the roots into the thick part
of $\G$. As a result we reduced the number
of boundary components of $T_1$.

Suppose the first possibility occurs.
Then we have decreased the number components of
$T_1 \cap \partial \Delta_j$ by $1$. We choose an inner most connected 
component once again. Eventually 
we will encounter possibility 2.
Then we split $T_1$ into two connected components
with a strictly smaller number of boundary components.
\end{proof}

\section{Opening Long Necks}\label{opennecks}
In this section, we prove the following:
\begin{prop}\label{movehandles}
Let $\Sigma$ be a strictly stable minimal two-sphere and let $\Gamma\subset N_{h_0}(\Gamma)$ be a two-sphere in $(\delta,k)$ canonical form for some $k>1$.  Then there exists an isotopy $\Gamma_t$ beginning at $\Gamma_0=\Gamma$ through surfaces with areas increasing by at most $\delta$ so that in some cell, the number of essential components of $\Gamma_1$ is fewer than $k$.
\end{prop}

The difficulty in Proposition \ref{movehandles} is that while the surface $\Gamma$ is in canonical form, there can be many wildly knotted, linked and nested arcs comprising the set of tubes.  In order to untangle this morass of tubes to obtain a vertical handle supported in a single ball requires the Light Bulb Theorem in topology, which we recall:

\begin{prop} [Light Bulb Theorem \cite{R}] \label{lb}
Let $\alpha(t)$ be an embedded arc in  $\mathbb{S}^2\times [0,1]$ so that $\alpha(0)=\{x\}\times\{0\}$ and $\alpha(1)=\{y\}\times\{1\}$ for some $x,y\in\mathbb{S}^2$.  Then there is an isotopy $\phi_t$ of $\alpha$ so that 
\begin{enumerate}
\item $\phi_0(\alpha)=\alpha$ 
\item $\phi_1(\alpha)$ is the vertical arc $\{x\}\times [0,1]$.
\end{enumerate}
\end{prop}

The Light Bulb theorem can be interpreted physically as that one can untangle a lightbulb cord hanging from the ceiling and attached to a lightbulb by passing the cord around the bulb many times.  

The simplest nontrivial case of Proposition \ref{movehandles} consists of two parallel spheres joined by a very knotted neck.  Here the Light Bulb Theorem \ref{lb} allows us to untangle this neck so that it is vertical and contained in one of the balls $B_i$. Thus the resulting surface is no longer in canonical form and we can iterate Step 1.  


We will in fact need the following generalization of the light bulb theorem (cf. Proposition 4 in  \cite{HT}):

\begin{prop}[Generalized Light Bulb Theorem]\label{general}
Let $M$ be a $3$-manifold and $\alpha$ an arc with one boundary point on a sphere component $\Gamma$ of $\partial M$ and the other on a different boundary component.  Let $\beta$ be a different arc with the same end points as $\alpha$.  Then if $\alpha$ and $\beta$ are homotopic, then they are isotopic.  

Moreover, if $\gamma$ is an arc freely homotopic to $\alpha$ (i.e. joined through a homotopy where the boundary points are allowed to slide in the homotopy along $\partial M$), then they are freely isotopic (i.e., they are joined by an isotopy with the same property).
\end{prop}

\noindent
\emph{Sketch of Proof:}
 The homotopy between $\alpha$ and $\beta$ can be realized by a family of arcs $\alpha_t$ so that $\alpha_t$ is embedded or has a single double point for each $t\in [0,1]$.  If $\alpha_{t_0}$ contains a double point, the curve $\alpha_{t_0}$ consists of three consecutive sub-arcs $[0,a]$, $[a,b]$, and $[b,1]$ so that without loss of generality $[0,a]$ connects to the two sphere $\Gamma$.  We can pull the arc $\alpha_{t_0}([b-\varepsilon,b+\varepsilon])$ transverse to $\alpha_{t_0}([0,a])$ along the arc $\alpha_{t_0}([0,a])$ and then pull it over the two sphere $\Gamma$, and then reverse the process.  We can then glue this deformation smoothly in the family $\alpha_t$ for $t$ near $t_0$ to obtain the desired isotopy.  The proof is illustrated in Figure \ref{fig:lightbulb}.

\qed



\begin{figure} 
   \centering	
	\includegraphics[scale=0.75]{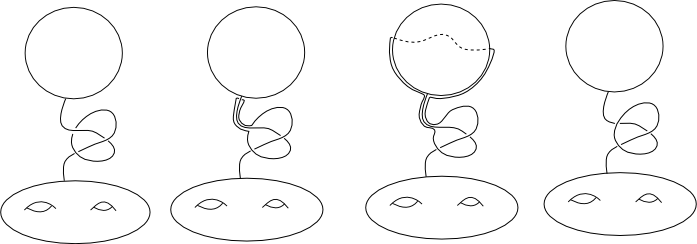}
	\caption{Changing the over-under crossing in the proof of the Light Bulb Theorem.}
	\label{fig:lightbulb}
\end{figure}

Let us now prove Proposition \ref{movehandles}:
\begin{proof}
Since $k>1$, we can find a cell $\Delta_k$ so that $\Gamma\cap\partial \Delta_k$ contains several small non-essential circles.  Let $C$ denote an innermost such circle.  By squeezing a small collar around $C$ we obtain a neck with two roots.  There are two connected components of $\Gamma\setminus C$.  Let us denote them $A$ and $B$.  There must be two consecutive sheets $S_1$ and $S_2$ comprising $\Sigma$ so that $S_1$ is contained in $A$ and $S_2$ is contained in $B$.  Thus we can move the roots of the collar about $C$ using Lemma \ref{lem: sliding} so that they are stacked on top of each other, one in $S_1$ and the other in $S_2$.  

By Proposition \ref{general} we can isotope the neck to then be a vertical neck contained in a single cell.  Thus the number of essential components has gone down by at least $1$ in this cell.

\end{proof}

 \subsection{Proof of Proposition \ref{main_deformation}.}

First we apply Proposition \ref{canoncial form}
to deform $\G$ so that it is $(\delta,k)$-stacked.
We consider two cases.

Case 1. $\Si$ is not homeomorphic to a sphere.
By assumption we have that the genus of every connected component
of $\G$ is less or equal to the genus of $\Si$.
It follows that each connected component of $\G$ either coincides
with a graphical sheet over $\Si$ (as more than one graphical sheet would imply that the genus of $\G$ is greater than that of $\Si$)
or is contained in a small ball of radius less than $\delta$.

Case 2. $\Si$ is homeomorphic to a sphere.
Suppose $\G$ has a connected component which
intersects more than one sheet $D_i$.
We apply Proposition \ref{movehandles}
reducing the essential multiplicity
of $\G$ in some cell $\Delta_i$.
Then we can apply Lemma \ref{stacking}
to reduce the area of the surface
by at least $c/2$. We iterate this procedure.
Eventually, every connected component
of $\G$ will either be graphical or
contained in a small ball.
This finishes the proof of 
Proposition \ref{main_deformation}.

  \section{Convergence to a surface $\Si$ with multiple connected
components}\label{multiplecomponents}
In this section, we generalize Proposition \ref{main_deformation}
to the situation where $\Si$ is disconnected and the area of
$\G$ outside of $N_h(\Si)$ is small. 

Let $\Si$ be an orientable connected surface
and let $p: N_h(\Si) \rightarrow \Si$
be the projection map.
Given a positive integer $m$ we will say that 
 a surface $\G$ has $\varepsilon$-multiplicity 
 $m$ if there exists a subset $U \subset \Si$
with $Area(U) < \varepsilon$ and for almost every
$x \in \Si \setminus U$ the set $\{ p^{-1}(x) \cap \G \}$
has exactly $m$ points.
Similarly, we will say that a surface $\G$ has $\varepsilon$-even 
(resp. $\varepsilon$-odd) multiplicity in $N_h(\Si)$
if there exists a subset $U \subset \Si$
with $Area(U) < \varepsilon$ and for almost every
$x \in \Si \setminus U$ the set $\{ p^{-1}(x) \cap \G \}$
has an even (resp. odd) number of points.

It is straightforward to check that
for all sufficiently small $h> 0$ and $\varepsilon>0$,
if $L(\G \cap \partial N_h (\Si)) \leq \frac{1}{100} \sqrt{\varepsilon}$,
then $\G$ is either $\varepsilon$-even or $\varepsilon$-odd in 
$N_h(\Si)$.

\begin{prop} \label{prop: multiple components}
Let $\Si = \sqcup \Si_k$, where each $\Si_k$
is a smooth strictly stable two-sided
connected minimal surface.

There exists $h_0, \varepsilon_0 >0$,
such that for all $h \in (0, h_0)$ and
$\varepsilon \in (0, \varepsilon_0)$ the following
holds. 

Suppose $\G$ satisfies 

(a) $Area (\G \setminus N_{h}(\Si))< \varepsilon$;

(b) $genus(N_h(\Si_i) \cap \G) \leq genus(\Si_i)$ for each $i$;

(c) $L(\G \cap \partial N_h (\Si)) \leq \frac{1}{100} \sqrt{\varepsilon}$.







Let $\{\Si_{k_j}\}$ denote the subset of minimal surfaces
for which $\G$ is $\varepsilon$-odd in $N_h(\Si_{k_j})$.
Then for every $\delta>0$ there exists an isotopy $\G_t$, such that:

(i) $Area(\G_t) < Area(\G) + \delta$

(ii) $\G_1$ is the union of $\{\Si_{k_j}\}$
connected 
by thin necks.

(iii) If $\{\Si_{k_j}\}$ is empty then $\G_1$
can be chosen to be a closed surface in a ball
of arbitrarily small radius.
\end{prop}



Note that assumption (b) is satisfied whenever
$\G$ is a strongly irreducible Heegaard splitting.

Proposition \ref{main_multiple} follows immediately from
Proposition \ref{prop: multiple components}.

%
%
%

\subsection{Reducing the area outside of the tubular
neighbourhood of $\Si$}

The following lemma is useful for reducing the area of thin hair (see \cite[Lemma 7.1]{Mo} for an analogue lemma in the context of Almgren-Pitts theory). 
\begin{lemma}[Reducing the area of thin hair]
 \label{contracting_hair}
Let $S$ be a surface in $M$
then for all sufficiently small 
$h>0$ the following holds.
There exists $\varepsilon(M, S, h)>0$ with the following
property. 
For every $\delta>0$ and every surface $\Si$ with 
$Area(\Si \setminus N_h(S))< \varepsilon$ their exists
a smooth isotopy $\Si_t$ with

(1) $\Si_0 =\Si$, 

(2) $Area(\Si_t) \leq Area(\Si)+ \delta$,

(3) $Area(\Si_1 \setminus N_{2h}(S))< \delta$,

\end{lemma}

\begin{proof}

We recall the following ``bounded path" 
version of the $\gamma$-reduction of \cite{MSY} 
used in the min-max setting of Simon-Smith  (see \cite[Section 7]{cd}). 
Let $\Sigma$ be an embedded surface in $M$, and $U$ be an open set included in $M$. Let $\mathfrak{Is}(U)$ be the set of isotopies of $M $ fixing $M\backslash U$, with parameter in $[0,1]$. For $\delta>0$ define
$$\mathfrak{Is}_{\delta}(U) = \{\psi\in\mathfrak{Is}(U); \mathcal{H}^2(\psi(\tau,\Sigma)) 
\leq \mathcal{H}^2(\tilde{\Sigma}) + \delta \text{ for all } \tau \in[0,1]\}.$$
An element of the above set is called a $\delta$-isotopy. Suppose that the sequence $\{\psi^k\}\subset \mathfrak{Is}_{\delta}(U)$ is such that 
$$\lim_{k\to\infty}\mathcal{H}^2(\psi^k(1,\Sigma)) = \inf_{\psi\in\mathfrak{Is}_{\delta}(U)}\mathcal{H}^2(\psi(1,\Sigma)).$$
Such a sequence is called minimizing. Then in $U$, $\psi^k(1,\Sigma)$ subsequently converges in the varifold sense to a smooth minimal surface $\hat{\Sigma}$.

Let us apply this $\gamma$-reduction with constraint to $U:=M\backslash \bar{N_{h}}(S)$. Let $\{\psi^k\}\subset \mathfrak{Is}_{\delta}(U)$ be a minimizing sequence. Then by the monotonicity formula for minimal surfaces, the area of $(M\backslash {N_{2h}}(S))\cap \psi^k(1,\Sigma)$ goes to zero so in particular for a $k'$ large enough, this area is smaller than $\delta$. The lemma is proved by taking $\Sigma_t = \psi^{k'}(t,\Sigma)$.

\end{proof}



%
%


The following lemma allows us to isotopically push 
discs into the neighbourhood of a surface
with almost no increase in the area.

\begin{lemma} \label{contracting_disc}
There exists $\varepsilon(M, \Si, h)>0$ with the following
property. 
Suppose $D \subset M \setminus N_{h}(\Si)$ 
is an embedded disc with 
$Area(D)< \varepsilon$, $l(\partial D) < \sqrt{\varepsilon}$ and $\partial D \subset \partial  N_{h}(\Si) $ 
There exists
a smooth isotopy $D_t$ with

(1) $D_0 =\G$, 

(2) $Area(D_t) \leq Area(D) + \delta$,

(3) $D_1 \subset  N_{2h}(\Si)$.
\end{lemma}

\begin{proof}
$\varepsilon < \frac{1}{100} r^2$.
Fix a triangulation of $M \setminus N_{h}(\Si)$ with each 3-simplex
$\Delta_j$ contained in a ball of radius $r/2$ centered at the barycenter of the simplex $p_j$.

Let $C(2)$ denote the union of the interiors
of 2-simplices in the triangulation.
After an initial deformation we may assume that
$D$ intersects $C(2)$ in a collection of small circles
and does not intersect the 1-skeleton or 0-skeleton
of the triangulation.
Indeed, by coarea formula we can find a radius $r_j \in (r,2r)$,
so that $\partial B_{r_j}(p_j) \cap D$ is a finite union of circles of total length less than $ 2 \frac{\epsilon}{r}$.
Using blow down - blow up  Lemma \ref{radial} we can deform $D$ to $D'$ so that 
$D'$ intersects each $2$-face of $Delta_j$ in a collection of circles
and does not intersect $1$-skeleton or $0$-skeleton
of $Delta_j$, and so that the area of $D'$ is
less than or equal to the area of $D$.
We perform a similar deformation for every simplex in
the triangulation.

We claim that there exists a 3-simplex $\Delta'$, so that
$\Delta' \cap D$ contains a connected component $D'$ diffeomorphic to a disc.
Given this claim we retract $D$ by inductively reducing the number
of connected components of $D \cap C(2)$.
Namely, we push out the disc outside of $\Delta'$ together with 
all connected components $D \cap \Delta'$ contained between
$D'$ and the $2$-simplex of $\Delta'$ that contains $\partial D'$.

To prove the claim, consider a tree where each vertex
corresponds to a 
connected component of $D \setminus C(2)$, and two vertices are connected 
by an edge if the corresponding connected components have a common boundary circle.
A terminal vertex of the tree must correspond to a disc.
\end{proof}

\begin{defi}
Let $\Si$ be connected and suppose
$\G \cap \partial N_h(\Si)$ is a collection
of small disjoint circles.
Let $C(\G)$ denote a closed surface in $N_h(\Si)$ obtained 
from $\G \cap N_h(\Si)$ by
capping each connected component of $\G \cap \partial N_h(\Si)$
with a small disc and perturbing to remove self-intersections.
\end{defi}

\begin{lemma}[Continuous cutting and gluing in the tubular neighborhood]\label{cutting and reattaching hair}

For every sufficiently small $\delta>0$
the following holds. Suppose $L(\G \cap \partial N_h(\Si)) \leq \delta$
and let $\{\G_t\}$ be an isotopy of closed surfaces with 
$\G_0 = C(\G)$. Then there exists an isotopy $\{\G'_t \}$, such that

(a) $\G'_0 = \G$;

(b) $\G'_t \setminus N_h(\Si) = \G \setminus N_h(\Si)$;

(c) $(\G'_t \cap N_h(\Si)) \setminus \G_t$ consists of 
a union of disjoint necks of total area $O(\delta^2)$.
\end{lemma}

\begin{proof}
Let $\gamma_1$ be an outermost connected component
of $\G \cap \partial N_h(\Si)$. Let $D_1$ and $D_2$ denote the two discs
in $C(\G)$ and $\partial N_h(\Si)$ respectively, corresponding to 
the surgery along $\gamma_1$. Let $\alpha_t$ be an embdedded arc with 
endpoints $p$ in $D_1$ and $q$ in $D_2$.
It is a consequence of standard topological theorems
(in particular, Cerf's theorem) 
that there exists an isotopy of embedded arcs $\alpha_t$,
which does not intersect $\G_t$, except at the endpoint $p_t \in \G_t$,
with $\alpha_0 = \alpha$ and the other endpoint equal to $q \in D_2 \subset \partial N_h(\Si)$.

If follows by compactness that a sufficiently small tube around $\alpha_t$
will be disjoint from $\G_t$ except at the root.
We glue in this family of necks to obtain a new isotopy of surfaces.
We proceed by induction on the number of connected components
of $\G \cap \partial N_h(\Si)$.
\end{proof}

We need one more lemma before we can
prove Proposition \ref{prop: multiple components}.

\begin{lemma}[Canonical form
in the presence of multiple connected components] \label{canonical multiple}

Suppose $\G$ satisfies the assumptions of
Proposition \ref{prop: multiple components}. 
Then the conclusions of Lemma \ref{contracting_hair}
hold and, moreover, we may also assume that 



(4) for each $i$, there exists an integer 
$m_i$ such that $\G_1$ has $\varepsilon$-multiplicity
$m_i$;

(5) if $genus(\Si_i) \geq 1$ then 
$m_i = 0$ or $1$;


(6) if $\G' \subset N_h(\Si_i)$
is a connected component of 
$\G_1 \cap  N_h(\Si_i)$, such that
$m_i>1$ and $\gamma \subset \partial \G'$
is an inner most circle in $\partial N_h(\Si_i)$,
then $\partial \G' = \gamma$.
\end{lemma}

\begin{figure} 
   \centering	
	\includegraphics[scale=0.8]{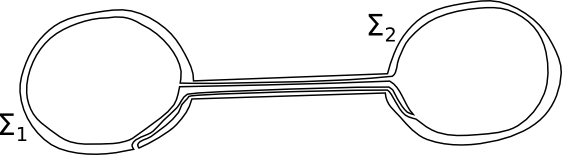}
	\caption{Surface $\G$ is within $\varepsilon$
	(in varifold norm) from $2\Si_1 +2 \Si_2$, where $\Si_1$ and 
	$\Si_2$ are stable minimal spheres. We can isotop
	$\G$ to a point while increasing its area by an arbitrarily small amount. This example illustrates that we will need to keep track of the surface outside of the tubular neighborhood to find a neck to open.  }
	\label{fig:dumbbell}
\end{figure}

\begin{proof}


By applying 
Proposition \ref{main_deformation} together with 
Lemma \ref{cutting and reattaching hair}
we may assume that in the tubular neighborhood
of each connected component $\Si_i$
surface $\G \cap N_h(\Si_i)$ looks like $m_i$ disjoint copies of
$\Si_i$ with thin necks that go into the boundary $\partial N_h(\Si_i)$.
This proves (4).
By assumption 
(b) of Proposition \ref{prop: multiple components} 
if $\Si_i$ has genus greater than or equal
to $1$ then $m_i = 0$ or $1$. This proves (5).


Finally, we can apply Lemma \ref{lem: sliding}
to slide tubes attaching to an inner most sphere,
so that condition (6) is satisfied.
\end{proof}

\subsection{Proof of Proposition \ref{prop: multiple components}.}
Assume that the part of the surface outside of small tubular neighborhood
of $\Si$ satisfies conclusions (1)-(6) of Lemmas \ref{contracting_hair} 
and \ref{canonical multiple}. 

By Proposition \ref{main_deformation} and 
Lemma \ref{cutting and reattaching hair}
we may assume that in the tubular neighborhood
of each connected component $\Si_i$
surface $\G \cap N_h(\Si_i)$ looks like $k$ disjoint copies of
$\Si_i$ with thin necks that go into the boundary $\partial N_h(\Si_i)$.
Also, if $\Si_i$ has genus greater than or equal to 
$1$ then $\G \cap N_h(\Si_i)$ is either empty or has one
connected component (which look like $\Si_i$ with necks 
escaping into the boundary of  $N_h(\Si_i)$).
Hence, we need to deal with multiple connected component only
in tubular neighbourhood of spherical components of
$\Si$.

Suppose $\G \cap N_h(\Si_i)$ has more than one
connected component.
Choose an inner most closed curve
$\gamma \subset \partial N_h(\Si_i) \cap \G$.
By squeezing a small collar that contains $\gamma$ we obtain a neck with two roots.   
There are two connected components $\Gamma \setminus \gamma$.  Let us denote them $A$ and $B$.  
Moreover, since we chose $\gamma$ to be innermost, by Lemma \ref{canonical multiple} (6),
exactly one sheet of $\G \cap N_h(\Si_i)$ is contained in $A$
and all other sheets are contained in $B$.
Hence, we can move the roots using Lemma \ref{lem: sliding} so that they are stacked on top of each other in the neighborhood of component $\Sigma_i$ and 
apply Proposition \ref{general}, so that in some cell on $\Sigma_i$, the number of essential components has gone down.  Thus we can reduce 
the number of sheets in $N_h(\Sigma_i)$ by at least two.   We proceed this way for as long as $\partial N_h(\Si_i) \cap \G$ consists of more than one connected component for any $i$.

Since the even/odd parity is preserved in this process
eventually we obtain that we have exactly one connected component
in the neighborhood for each $\Si_{k_j}$ where $\G$ had
$\varepsilon$-odd multiplicity.



\subsection{Gluing two families of isotopies
and interpolation.}

Proposition \ref{main_multiple} implies that we can deform 
the surface converging to a minimal surface with multiplicity
so that it has form
$\Si = \Si_{thin} \sqcup \Si_{thick}$, where
$\Si_{thick}$ is a disjoint union of embedded stable 
surfaces $\Si_i$ with $\delta$-size discs removed 
and $\Si_{thin}$ is contained in the boundary of 
a tubular neighbourhood of an embedded graph.

Given two isotopic surfaces with such a decomposition,
and assuming some additional topological conditions,
we would like to find an isotopy between them
that fixes $\Si_{thick}$.

\begin{prop} \label{joining two sides}
Suppose $\G_1$ and $\G_2$ are two isotopic embedded surfaces
in $M$ satisfying assumptions of Proposition \ref{prop: multiple components}
and, moreover, suppose the set of stable connected components
$\Si_i$, where $\G_j$, $j=1,2$, is $\epsilon$-odd coincides.
Assume, in addition, that $M = S^3$ or is a lens space
and $\G_j$ is a strongly irreducible 
Heegaard splitting of $M$ diffeomorphic to a torus or 2-sphere.

Then for every $\delta>0$ there exists an isotopy between  
$\G_1$ and $\G_2$ of area at most
$\max \{Area(\G_1), Area(\G_2) \}+\delta$.
\end{prop}

\begin{proof}
By Proposition \ref{prop: multiple components},
$\G_j$ can be deformed into a union of stable
spheres connected by thin necks. 

Suppose first that $\G_j$ is a 2-sphere.
Then by results of sections 4 and 5
any two such configurations are isotopic 
with arbitrarily small area increase.

Suppose that $\G_j$ is a 2-torus.
If $\Si_i$ has a torus component,
then by applying Proposition \ref{general}
we can find an isotopy between thin parts
of $\G_1$ and $\G_2$ through surfaces of small area.
If the thick parts of $\G_1$ and $\G_2$ 
consist of spheres, we can use Lemma
\ref{lem: sliding} to deform $\G_j$
so that it consists of stable spheres
connected consecutively by thin loops
with the last sphere connected to a thin torus $T_j$
(that is, boundary of a $\delta$-neighbourhood
of an embedded loop).
By strong irreducibility assumption
each sphere $\Si_i$ bounds a ball on one of the sides.
In particular, thin tori $T_1$ and $T_2$ are
isotopic in the complement of $\Si_i$'s.
The desired isotopy can then be obtained
by repeated application of Proposition \ref{general}.
\end{proof}

We conjecture that the analogue of this Proposition
holds when $\G_1$ and $\G_2$ are strongly irreducible
Heegaard splittings of a 3-manifold $M$.

  \section{Deformation Theorems and Index bounds}\label{deformation1}
    
  In this section we apply our Interpolation result to obtain the Deformation Theorem.  Then we show how the Deformation Theorem easily implies the index bound Theorem \ref{main}. 
  
First we recall the following lemma from Marques-Neves (Corollary 5.8 in \cite{mn4}) expressing the fact that strictly stable surfaces are isolated from other stationary integral varifolds:

\begin{lemma}[Strictly Stable Surfaces Are Isolated] \label{isisolated}
Suppose $\Sigma$ is strictly stable and two-sided.  Then there exists $\varepsilon_0:=\varepsilon_0(\Sigma)$ so that every stationary integral varifold $V\in\mathcal{V}_n(M)$ in $B^F_{\varepsilon_0}(\Sigma)$ coincides with $\Sigma$.
\end{lemma}

\begin{rmk}
Note that the assumption of stability is essential in Lemma \ref{isisolated}.  For example, in round $\mathbb{S}^3$ there exist a sequence of minimal surfaces $\Sigma_g$ converging to twice the Clifford torus $C$.  Thus in any small neighborhood of $2C$ are many stationary integral varifolds (\cite{KMN}, \cite{KY}).

\end{rmk}												

The proof of the lemma follows from the application of the squeezing map $P_t$
(see  section \ref{sec:squeezing}).

Let us record some further properties of the map $P_t$ (and sometimes we will write $P_t(x)$ as $P(t,x)$).
\begin{lemma}[Projection Map (Proposition 5.7 \cite{mn4})]\label{properties}
There exists $r_0>0$ such that $P: N_{r_0}\times [0,1]\rightarrow N_{r_0}$ satisfies:
\begin{enumerate}
\item $P(x,0)=x$ for all $x\in N_{r_0}$, $P(x,t)=x$ for all $x\in S$ and $0\leq t\leq 1$
\item $P(N_{r_0},t)\subset N_{r}$ for all $0\leq t\leq 1$ and $r\leq r_0$, and $P(\Sigma_{r_0},1)=S$
\item the map $P_t$ is a diffeomorphism onto its image for all $0\leq t <1$ 
\item for all varifolds, and every connected component $N$ of $N_{r_0}$, the function $t\rightarrow ||P_t(V)||_{N}$ has strictly negative derivative unless $V$ is supported on $S$, in which case it is constant.
\end{enumerate}

\end{lemma}

 In the following, we consider a minimizing sequence $\Phi_i$ of sweepouts which means that
  \begin{equation}
 \sup_{t\in [0,1]} \mathcal{H}^2(\Phi_i)\rightarrow W.
 \end{equation}
 
 We have the following result (similar to Deformation Theorem C in Section 5.6 in \cite{mn4}) which allows us to deform the sweep-out $\Phi_i$ away from stable surfaces with multiplicities in such a way that no new stationary integral varifolds of mass $W$ arise as limits of min-max sequences:
\begin{thm}[Deformation Theorem]\label{deformation}
Suppose $M = S^3$ or a lens space. 
Suppose $\Phi_i$ is a pulled tight sequence of sweepouts by
spheres or tori (so that any min-max sequence obtained from it has a stationary limit).   Let $\Sigma$ be a stationary varifold such that 
\begin{enumerate}
\item The support of $\Sigma$ is a strictly stable two-sided closed embedded minimal surface $S$ (potentially disconnected) so that 
\begin{equation}
\Sigma=\sum_{i=1}^k n_i S_i,
\end{equation}
where $n_i$ are positive integers
\item $W = |\Sigma|$.
\end{enumerate}
Then there exist $\eta>0$ and $j_0\subset\mathbb{N}$ so that for all $i\geq j_0$, one can find a non-trivial sweepout $\Psi_i$ 
so that
\begin{enumerate}
\item $\limsup_{i\rightarrow\infty} \sup_{t\in [0,1]}(||\Psi_i(t)||)\leq W$
\item Any min-max sequence obtained from $\Psi_i$ converges to either (a) the limit of a min-max sequence from $\Phi_i$ and is disjoint from an $\eta$-ball in the $\mathbf{F}$-metric around $\Si$ or else (b) disjoint from an $\eta$-ball in $\mathbf{F}$-metric about any stationary integral varifold of mass $W$.
\end{enumerate}
\end{thm}
\begin{proof}
Let $\alpha>0$ be a small number to be specified later.  If $\alpha$ is small enough, we can find finitely many disjoint intervals $V'_{i,\alpha}\subset [0,1]$ so that if $t\in [0,1]\setminus V'_{i,\alpha}$ then $\mathbf{F}(\Phi_i(t),\Sigma)\geq\alpha$ and for $t\in V'_{i,\alpha}$, we have $\mathbf{F}(\Phi_i(t),\Sigma)\leq 2\alpha$.  Moreover, for $t\in\partial V'_{i,\alpha}$, we obtain $\alpha\leq \mathbf{F}(\Phi_i(t),\Sigma)\leq 2\alpha$.  Note that for $\alpha$ small enough, it cannot happen that $V'_{i,\alpha}=[0,1]$ since $\Phi_i([0,1]$ is a non-trivial sweepout of $M$.   Moreover, it cannot happen that $V'_{i,\alpha}$ is empty (as otherwise the deformation theorem would already be proved).   Thus the set $V'_{i,\alpha}$ consists of several sub-intervals of $[0,1]$.  Let $V_{i,\alpha}$ denote one connected component of $V'_{i,\alpha}$.  We will amend the sweepout $\Phi_i$ in the interval $V_{i,\alpha}$ and since $V'_{i,\alpha}$ consists of disjoint such intervals, we can apply the analagous alteration in each connected component.

We can choose $\alpha$ so small so that if $\mathbf{F}(\Phi_i(t),\Sigma)<\alpha $ then the area of $\Phi(t)$ outside of $N(\Sigma)$ is smaller than the $\varepsilon$ in the statement of Lemma \ref{contracting_hair},
the $\varepsilon_0$ in the statement of Proposition 
\ref{prop: multiple components},
and less than $3\varepsilon_0$ (where $\varepsilon_0$ is from Lemma \ref{isisolated}).

Let us consider one such interval $V_{i,\alpha}$, which is $[t_a,t_b]$.
 
For each $i$, Lemma \ref{contracting_hair} furnishes an isotopy $H^1_i(t)_{t\in[0,1]}$ so that
\begin{enumerate}
\item $H^1_i(0)=\Phi_i(t_a)$
\item $H^1_i(t)\leq |\Phi_i(t_a)|+1/i$
\item $H^1_i(1)$ is a surface with $|H^1_i(t)\cap N(\Si)|<1/i$
\end{enumerate}

We can further let $H^2_i(t) = P_t(H^1_i(1))$ where $t$ is parameterized from $t=0$ to $t=1-q_i$ (where $q_i$ will be a sequence approaching zero as $i\rightarrow 0$ to be chosen later). 

Furthermore, when $i$ is large enough it follows that from Proposition 
\ref{prop: multiple components}
there exists an isotopy $H_i^3(t)_{t\in[0,1]}$ so that

\begin{enumerate}
\item $H^3_i(0)=H^2_i(1-q_i)$
\item $H^3_i(t)\leq |H^3_i(1)|+1/i$
\item $H^3_i(1)  =\bigcup_{i\in\mathcal{C}}S'_i\cup\bigcup_k T^1_k$
\end{enumerate}
 
 In (3), $\mathcal{C}$ denotes the subcollection of $\{1,2,...k\}$ such that $S_i$ is $\varepsilon$-odd.   $S'_i$ denotes the surface $S_i$ with several disks removed, and the set $\{T^1_i\}$ consists of thin tubes connecting to the the various $S'_i$ at the circles where the disks from $S_i$ have been removed so that $\bigcup_{i\in\mathcal{C}}S'_i\cup\bigcup_k T^1_k$ is a closed surface.

Let us define the isotopy $H^4_i(t)_{t\in[0,1]}$ by concatenating  $H^1_i(t)_{t\in[0,1]}$, $H^2_i(t)_{t\in[0,1-q_i]}$ and $H^3_i(t)_{t\in[0,1]}$

Thus $H^4_i(t)_{t\in[0,1]}$ furnishes an isotopy between $\Phi_i(t_a)$ and the closed surface $\bigcup_{i\in\mathcal{C}}S'_i\cup\bigcup_k T^1_k$.  

Similarly, we can obtain an isotopy $H^5_i(t)_{t\in[0,1]}$ beginning at $\Phi_i(t_b)$ when $t=0$ and terminating at $\bigcup_{i\in\mathcal{C}}S'_i\cup\bigcup_k T^2_k$ satisfying the hypothesis above where the thin tubes $T^1_k$ are replaced by potentially different $T^2_k$.  


By Proposition \ref{joining two sides} there is an isotopy connecting $\bigcup_{i\in\mathcal{C}}S'_i\cup\bigcup_k T^2_k$ and $\bigcup_{i\in\mathcal{C}}S'_i\cup\bigcup_k T^1_k$ that increases area at most $1/i$ along the way.  

Denote the isotopy $I_i(t)_{t\in [0,1]}$ that begins at $t=0$ at $\bigcup_{i\in\mathcal{C}}S'_i\cup\bigcup_k T^1_k$ and terminates at $t=1$ at $\bigcup_{i\in\mathcal{C}}S'_i\cup\bigcup_k T^2_k$. 

Thus we can concatenate $H^4_i(t)_{t\in[0,1]}$, together with $I_i(t)_{t\in [0,1]}$ and then $H^5_i(1-t)_{t\in[0,1]}$ to obtain an isotopy $H^6_i(t)_{t\in[0,1]}$ beginning at $t=0$ at  $\Phi_i(t_a)$ and terminating at $t=1$ at  $\Phi_i(t_b)$. 

Let us replace by the sweepout $\Phi_i(t)$ in the interval $[t_a,t_b]$ by the isotopy $H^6_i(t)_{t\in[0,1]}$.  In this way we obtain a new sweep-out $\Phi'_i(t)$.  

We claim that $\Phi'_i(t)$ satisfies the conclusion of the theorem.   It follows from the properties of $H^1_i$, $H^2_i$, $H^3_i$ and their mirror images $H^5_i$ that
\begin{equation}
\sup_{t\in [0,1]} |\Phi'_i(t)| \rightarrow W\mbox{ as } i\rightarrow\infty \end{equation}

We need to show that \emph{no min-max sequence} of $\Phi'_i(t)$ obtained from the intervals $V_{i,\varepsilon}$ converges to anything in an $\varepsilon$-ball about the space $\mathcal{S}$ of stationary integral varifolds with mass $W$.  Let us denote this latter space $B_\varepsilon(\mathcal{S})$.  
\\
\\
\noindent
\emph{Claim 1:} There exists an $\varepsilon_1>0$ so that for $i$ large enough,  $H^1_i(t)\cap B_{\varepsilon_1}(\mathcal{S})=\emptyset$ for all $t\in [0,1]$.
 \\
 \\
 \noindent
 Suppose toward a contradiction that there is a subsequence (not relabelled) $i\rightarrow\infty$ as well as a sequence of stationary integral varifolds $V_i$ as well as $t_i\in [0,1]$ so that $\mathbf{F}(V_i, H^1_i(t_i))\rightarrow 0$ as $i\rightarrow\infty$.  We can assume $|V_i|=W$ and that $V_i$ converge to a stationary integral varifold $V$.  Note that for $i$ large enough,
 \begin{equation}
 \mathbf{F}(H^1_i(t_i),\Sigma)<\mathbf{F}(H^1_i(t_i),\Phi_i(t_a))+\mathbf{F}(\Phi_i(t_a),\Sigma) \leq 1/i+2\varepsilon <\varepsilon_0 
 \end{equation}
 because of item (3) in the list of properties that $H^1_i$ satisfies and because the $\mathbf{F}$ metric between two surfaces is bounded from above by the area of the symmetric difference of the surfaces. 
 Thus $V_i\in B_{\varepsilon_0}^\mathbf{F}(\Sigma)$ for $i$ large and so by Lemma \ref{isisolated} we have that for $i$ large enough, $V=V_i=\Sigma$.   Thus we have $\mathbf{F}(H^1_t(t_i),\Sigma)\rightarrow 0$.
   From this we can easily deduce that $\Phi_i(t_a)\rightarrow\Sigma$, which contradicts the fact that $\mathbf{F}(\Phi_i(t_a),\Sigma)=\alpha$ by definition.  Indeed, by construction the limit $L$ of $\Phi_i(t_a)$ and $L'$ of $H^1_t(t_i)$ coincide in $N(\Sigma)$ with total mass $W$, and moreover $L$ can have no support outside of $N(\Sigma)$ as $\Phi_i$ are a minimizing sequence with maximal area approaching $W$ as $i\rightarrow\infty$.
\\
\\
 \noindent
 \emph{Claim 2:}  At least one of the integers $\{n_i\}_{i=1}^k$ is greater than $1$.  Moreover, $$H^2_i(t)\cap B_{\varepsilon_2}(\mathcal{S})=\emptyset$$ for all $t\in [0,1]$ and some $\varepsilon_2>0$.\\
 \\
 Suppose toward a contradiction that $n_1=n_2=...=n_k=1$.   Consider the sequence of surface $H^1_i(1)$, and their limit $A$ which is an integral varifold supported in the tubular neighborhood $N(\Sigma)$ that is homologous to $\Sigma$.  We can also consider $\tilde{A}=P_1(A)$.  Note that $|A|\leq W$.  It follows from the properites of the projection mapping $P_t$ that $|\tilde{A}|\leq|A|$.  But $W\leq |\tilde{A}|$ since $\tilde{A}$ is an integral varifold supported on $\Sigma$ with multiplicity at least one everywhere.  Thus we obtain $|\tilde{A}|=|A|=W$ which implies by item (4) in Lemma \ref{properties} that $A=\Sigma$.  Thus $H^1_i(1)\rightarrow \Sigma$.  
 Since $\mathbf{F}(H^1_i(1),\Phi_i(t_a))\rightarrow 0$, we obtain as in Claim 1 that  $F(\Phi_i(t_a), \Sigma)\rightarrow 0$.  But this contradicts the definition of $t_a$ as in Claim 1.
 
For the second part of the claim, suppose toward a contradiction that there is a subsequence (not relabelled) $i\rightarrow\infty$ as well as a sequence of stationary integral varifolds $V_i$ as well as $t_i\in [0,1]$ so that $\mathbf{F}(V_i, H^2_i(t_i))\rightarrow 0$ as $i\rightarrow\infty$.  We can assume $|V_i|=W$ and that $V_i$ converge to a stationary integral varifold $V$ and that $t_i\rightarrow t$.  Note that $V$ is supported on $S$ by Property (4) of the projection map and the fact that the area of $H^2_i(t_i)$ outside of $N(\Sigma)$ is approaching $0$ as $i\rightarrow \infty$.  Letting $A$ denote the limit of the surfaces $H^1_i(1)$, we have $|A|\leq W$.  We have from the definition of $H^2_i$ it follows that $V=P_t(A)$ and $|V|\leq |A|$ as the map $P_t$ is area non-increasing. Thus we obtain $|V|=|A|=W$ from which we deduce that $H^1_i(1)=V$, contradicting Claim 1.\\
 \\
 \noindent
 \emph{Claim 3:}  There exists $\varepsilon_2>0$ so that for $i$ large enough, and a choice of $q_i$ we have that for some connected component $S_1$ of $S$, there holds 
 \begin{equation}
 |H^2_i(1-q_i)\cap N(S_1)| < |\Sigma| -\varepsilon_2.
 \end{equation}
We will in fact show that 
\begin{equation}\label{contra}
|P_1(H^1_i(1))\cap N(S_1)| \leq |\Sigma|-\varepsilon_2,
\end{equation}
 from which Claim 3 follows immediately by choosing $q_i$ small enough.  If \eqref{contra} were to fail we would have a subsquence of $i$ so that for all $k$,
 \begin{equation}\label{issupported}
 \limsup_{i\rightarrow\infty} |P^k_1(H^1_i(1))|\geq |\Sigma|(N(S_k))
 \end{equation}
 
 Let us assume $P_1(H^1_i(1))$ converges as varifolds to a varifold $V$ with $|V|=W$ and $V$ is supported in $S$.  Arguing as in the previous claims, we conclude that $H^1_i(1)\rightarrow V$ and moreover that $\Phi_i(t_a)\rightarrow V$.  By \eqref{issupported}, the fact that $|V|=W$, and the Constancy Theorem, we obtain that $V$ is equal to $\Sigma$.   So we have $\Phi_i(t_a)\rightarrow \Sigma$ contradicting the fact that $\mathbf{F}(\Phi_t(t_a),\Sigma)=\alpha$.
\\
\\
 \noindent
 \emph{Claim 4:}  There exists an $\varepsilon_2>0$ so that for $i$ large enough,  $H^3_i(t)\cap B_\varepsilon(\mathcal{S})=\emptyset$ for all $t\in [0,1]$.
 \\
 \\
As in Claim 1), suppose toward a contradiction that there is a subsequence (not relabelled) $i\rightarrow\infty$ as well as a sequence of stationary integral varifolds $V_i$ as well as $t_i\in [0,1]$ so that $\mathbf{F}(V_i, H^3_i(t_i))\rightarrow 0$ as $i\rightarrow\infty$.  We can assume $|V_i|=W$ and by Allard's compactness theorem $V_i$ converge to a stationary integral varifold $V$.  Thus also $H^3_i(t_i)$ converge to $V$. 
The only stationary integral varifolds supported in $N(S)$ are the $S_i$ with some integer multiplicities.  It follows that $V=\sum k_i S_i$ for some non-negative integers $k_i$.     Note that since $|V|=W$, and by claim 3), $V\neq \Sigma$, it follows that at least one of the $k_i$ is less than $n_i$, and at least one of the $k_i$ is is greater than $n_i$ (say $k_r$).  Thus $H^3_i(t_i)\cap N(S_r)$ converges to $k_r S_r$.  But this as before implies $\Phi_i(t_a)\cap N(S_r)$ converges to $k_r S_r$, which contradicts the fact that $\mathbf{F}(\Phi_i(t_a), \Sigma)\geq\alpha$.
 \\
 \\
 Since $F(\Phi'(t),\Sigma)\geq \alpha$ for $t\in [0,1]\setminus V'_{i,\alpha}$, by setting $\eta:=\min(\varepsilon_1,\varepsilon_2,\alpha)$, Claim (3) and Claim (1) together imply the Deformation theorem. 
 
 Finally, we observe that family $Psi_i$ forms 
 a non-contractible loop in the space of flat cycles. 
 This follows by Interpolation results in Appendix A
 of \cite{mn4}. In particular, the sweepout that
 we obtained is non-trivial.
\end{proof}
\noindent
Given the Deformation Theorem, together with Deformation Theorem A of Marques-Neves \cite{mn4} it is easy to prove the Index Bound in the Simon-Smith setting which we restate:
\begin{thm}[Index Bounds for Minimal Two-Spheres] \label{main2}
Let $M$ be a Riemannian three-sphere with bumpy metric. Then a min-max limit of minimal two-spheres $\{\Gamma_1, ...\Gamma_k\}$ satisfies:
\begin{equation}\label{generic}
\sum_{i=1}^k\text{\emph{index}}(\Gamma_i) = 1.
\end{equation}
If the metric is not assumed to be bumpy, then we obtain
\begin{equation}\label{nongeneric}
\sum_{i=1}^k\text{\emph{index}}(\Gamma_i) \leq  1 \leq 
\sum_{i=1}^k\text{\emph{index}}(\Gamma_i)+\sum_{i=1}^k\text{\emph{nullity}}(\Gamma_i).
\end{equation}\end{thm}
\begin{proof}
 Suppose $\{\Phi_i\}$ is a pulled-tight minimizng sequence.  Since the metric is bumpy, there are only finitely many embedded index $0$ orientable minimal surfaces with area at most $W$.   Thus there are only finitely many stationary integral varifold supported on a strictly stable minimal surface with total mass $W$.  Denote this set $\mathcal{W}_0$.  Applying the Deformation Theorem \ref{deformation} we obtain a new minimizing family $\{\Phi'_i\}$ 
 so that no element of $\mathcal{W}_0$ is in the critical set of $\{\Phi'_i\}$, while no new stationary integral varifolds are in the critical set.  After pulling this family tight, we can apply Theorem 6.1 (\cite{mn4}) we can find a minimal surface in the critical set of $\{\Phi''_i\}$ with index at most $1$.  Thus the Theorem is proved. 
 
 If the metric $g$ not bumpy, then we can take a sequence of metrics $g_i$ converging to $g$. The min-max limit $\Lambda_i$ for each $g_i$ can be chosen to have index $1$ by the above.  From convergence of eigenvalues of the stability operator we have that, if the convergence is with multiplicity $1$,  the stable components can converge to a stable minimal surface in $g$ with or without nullity, and the unstable component may converge to either a) an unstable minimal surface or else b) a stable minimal surface with nullity.  This establishes the bounds \eqref{nongeneric} in these cases.  If the convergence is instead with multiplicity for some component, the limit is automatically stable as it has a positive Jacobi field, establishing \eqref{nongeneric}.  Finally note that in the case of collapse with multiplicity the genus bounds \eqref{gb} are preserved. 
 \end{proof}

 We will also need to consider the situation of widths of manifolds with boundary.  To that end, let $M$ be a three-manifold and $\Sigma$ a strongly irreducible splitting so that $M\setminus\Sigma=H_1\cup H_2$.  Suppose $\Sigma^1_k<...<\Sigma^1_0=\Sigma$ is a chain of compressions on the $H_1$ side and $\Sigma^2_j<...<\Sigma^2_0=\Sigma$ is a chain of compressions on the $H_2$ side.  

Remove from $M$ the handlebodies bounded by the components of $\Sigma^1_k$ and $\Sigma^2_j$ to obtain a manifold $\tilde{M}$ with boundary $B_1$ in $H_1$ and $B_2$ in $H_2$.   Suppose $B_1$ and $B_2$ are strictly stable minimal surfaces.   Consider sweepouts $\{\Sigma_t\}$ beginning at $B_1$ together with arcs joining the components,  and terminating at $B_2$ together with arcs.  We moreover demand that $\Sigma_t$ is isotopic to $\Sigma$ for $0<t<1$.   Let $W$ denote the width for this min-max problem.  Then we have the following proposition, which allows us to apply min-max theory:

\begin{prop}[Boundary Case]\label{boundary}
\begin{equation}
W>\sup_{C\in B_1\cup B_2} |C|.  
\end{equation}
Thus by the index bound \eqref{generic2} we obtain in the interior of $\tilde{M}$ an index $1$ minimal surface. 
\end{prop}

\begin{proof}
Suppose toward a contradiction $W=\sup_{C\in B_1\cup B_2} |C|$.  Suppose without loss of generality that the supremum is realized by $B_1$.   Let us consider a pulled-tight minimizing sequence $\{\Sigma^i_t\}_{i=1}^\infty$ so that \begin{equation}
\sup_{t\in [0,1]} |\Sigma^i_t|\rightarrow W\mbox{ as } i\rightarrow\infty
\end{equation}

Let $N:=N_{r_0}$ denote the tubular neighborhood about $B_1$ on which $P_t$ is defined via Lemma \ref{properties}.  For all $\varepsilon>0$ there exists $\delta>0$ so that if $\mathbf{F}(X,B_1)<\delta$ then $\mathcal{H}^2(X\setminus N))<\varepsilon$.    Choose $\varepsilon_0$ to be that provided by Lemma \ref{contracting_hair}.

Since $\Sigma^i_t$ converges as varifolds to $B_1$ as $t\rightarrow 0$, for each $i$, we can choose $t_i$ so that $F(\Sigma^i_{t_i},B_1)=\delta/2$ and $\mathcal{H}^2(\Sigma^i_{t_i}\setminus N)<\varepsilon_0$.  Thus from Lemma \ref{contracting_hair} we can isotope $\Sigma^i_{t_i}$ to a surface satisfying $\mathcal{H}^2(\tilde{\Sigma}^i_{t_i}\setminus N)<\delta_i$ where $\delta_i$ is any sequence approaching $0$ as $i\rightarrow\infty$.

Since $\delta_i$ is approaching zero, we can perform neck-pinches on $\partial N\cap \tilde{\Sigma}^i_{t_i}$ (c.f. Proposition 2.3 in \cite{dp}) o obtain a sequence of surfaces $ \overline{\Sigma}^i_{t_i}$ entirely contained in $N$ so that
\begin{equation}\label{long}
 |\overline{\Sigma}^i_{t_i}|\leq |\tilde{\Sigma}^i_{t_i}|\leq |\Sigma^i_{t_i}|.
\end{equation}  
Let $A$ denote a subsequential limit of $\overline{\Sigma}^i_{t_i}$ as $i\rightarrow \infty$.  Note that $|A|\leq W$ by \eqref{long} and the fact that $\Sigma^i$ are a minimizing sequence.   Let $\tilde{A}=P_1(A)$.  We have $W\leq |\tilde{A}|$ since $\tilde{A}$ is an integral varifold supported on $B_1$ with multiplicity at least one everywhere.  By the properties of $P$ (Lemma \ref{properties}) we have $|\tilde{A}|\leq |A|$.  Thus putting this together we obtain $W\leq |\tilde{A}|\leq |A|\leq W$.  In other words, $\tilde{A}=A$ again by the properties of the projection map (Lemma \ref{properties}).  It follows that $A$ is supported on $B_1$ and thus must consist of the surfaces comprising $B_1$ with multiplicity $1$ by the Constancy Theorem since $A$ is stationary (as the limit of a min-max sequence from a pulled-tight minimizing sequence).  Thus we have $\Sigma^i_{t_i}\rightarrow B_1$ with multiplicity $1$.  But this contradicts the fact that $F(\Sigma^i_{t_i},B_1)=\delta/2$ for all $i$.

\end{proof}

  \section{Applications}\label{applications}

In this section, let us prove the claim of Pitts-Rubinstein (Theorem \ref{pr}), which we restate:
 \begin{thm}[Pitts-Rubinstein Conjecture (1986)] \label{pr2}
 Let $M$ be a hyperbolic $3$-manifold and $\Sigma$ a strongly irreducible Heegaard surface.  Then either 
 \begin{enumerate}
 \item $\Sigma$ is isotopic to a minimal surface of index $1$ or $0$ or 
 \item after a neck-pinch performed on $\Sigma$, the resulting surface is isotopic to the boundary of a tubular neighborhood of a stable one sided Heegaard surface.
 \end{enumerate}
 If $M$ is endowed with a bumpy metric, in case (1) we can assume the index of $\Sigma$ is $1$.
  \end{thm}
  
We also  prove Theorem \ref{pr2} under the assumption $M$ is a lens space not equal to $\mathbb{RP}^3$:
\begin{thm}[Heegaard tori in lens spaces]\label{lensspaces}
If $M\neq\mathbb{RP}^3$ is a lens space, then $M$ contains a minimal index $1$ or $0$ torus.  
\end{thm}

We will use repeatedly the following essential estimate due to Schoen \cite{s}:
\begin{prop}[Curvature Estimates for Stable Minimal Surfaces]
Let $M$ ba three-manifold.  Then there exists $C>0$ (depending only on $M$) so that if $\Sigma$ is a stable minimal surface embedded in $M$, then
\begin{equation}\label{bounds}
\sup_{x\in\Sigma} |A|^2\leq C.
\end{equation}
\end{prop}

We will need the following non-collapsing lemma:
\begin{lemma}[Non-collapsing]\label{volumelowerbound}
Let $M$ be a closed Riemannian three-manifold.  For all $C>0$ there exists $\varepsilon(C,M)>0$ so that if $\Sigma\subset M$ is a closed embedded two-sided minimal surface bounding a region $H$ that is not a twisted $I$-bundle over a non-orientable surface and satisfying
\begin{equation}
\sup_{x\in\Sigma} |A|\leq C, 
\end{equation}
then the volume of $H$ is at least $\varepsilon(C,M)$.
\end{lemma}
\begin{rmk}
Note that the assumption on the topology of $H$ is essential.  In $\mathbb{RP}^3$ one can easily find metrics in which a sequence of stable two-spheres converges smoothly with bounded curvature to $\mathbb{RP}^2$ with multiplicity $2$. These stable two spheres bound a twisted interval bundle about $\mathbb{RP}^2$ with volumes approaching zero. Moreover, the assumption of bounded curvature is essential, as the doublings of the Clifford torus give an example of a sequence of minimal surfaces bounding volumes approaching zero. 
\end{rmk}
\begin{proof}
Suppose the lemma fails.  Thus for some $C>0$ there is a sequence $\Sigma_i$ of embedded minimal surfaces satisfying
\begin{equation}\label{curvbound}
\sup_{i\in\mathbb{N}} \sup_{x\in\Sigma_i} |A|\leq C, 
\end{equation}
where $R_i$ is a handlebody bounded by $\Sigma_i$ with $\mbox{vol}(R_i)\rightarrow 0$.

If the area of $\Sigma_i$ are uniformly bounded, then from the curvature bound we obtain that $\Sigma_i$ converges with multiplicity one to a closed embedded minimal surface $\Sigma$ or else with multiplicity $2$ to a non-orientable surface $\Gamma$ that is a one-sided Heegaard splitting.  
In the first case, since $\Sigma$ bounds a definite volume on both sides, the smooth convergence implies that $\Sigma_i$ do as well.  In the second case we still have $\mbox{vol}(R_i)\rightarrow\mbox{vol}(M\setminus\Gamma)>0$.  

Thus we can assume that the areas of $\Sigma_i$ are unbounded.  Again because of the curvature bound \eqref{curvbound}, upon passing to a subsequence, we can assume $\Sigma_i$ converges to a smooth minimal lamination $\mathcal{L}$.  

It follows that one can find a covering $\{B_j\}$ of $M$ with the property that in any ball, there's a diffeomorphism $\phi_j:B_j\rightarrow D^2\times [0,1]$ so that $\phi_j(\mathcal{L})$ consists of $D^2\times K$, where $K$ is a closed subset of $[0,1]$.  For each $i$, we can consider $\phi_j(\Sigma_i)$ which consist of a disjoint union of several graphs $G^{i,j}_1,..., G^{i,j}_{r_i}$ (where, potentially $r_i\rightarrow\infty$ as $i\rightarrow\infty$).  Let $A^{i,j}_k$ denote the region between $G^{i,j}_k$ and $G^{i,j}_{k+1}$.  Note that $A^{i,j}_k\subset R_i$ for all $k$ odd.  Moreover, since $\mbox{vol}(R_i)\rightarrow 0$, it follows that $\mbox{vol}(A^{i,j}_k)\rightarrow 0$ for $k$ odd.  Thus we obtain that
\begin{equation}
|G^{i,j}_{k}-G^{i,j}_{k+1}|\rightarrow 0 \mbox{ as } i\rightarrow\infty.
\end{equation}
We can find a graph $G^{i,j}_{k+1/2}$ between $G^{i,j}_{k}$ and $G^{i,j}_{k+1}$ and thus in $B_i$ we can retract $\Sigma_i$ to smoothly with multiplicity $2$, to the graph $G^{i,j}_{k+1}$.   Since we have such retractions in all balls $B_i$, by gluing these together we obtain a retraction of $\Sigma_i$ onto a closed embedded surface $\Gamma$.   But a connected surface cannot retract to a closed surface $\Gamma$ smoothly with multiplicity two unless $\Gamma$ is non-orientable.  Thus we obtain that $R_i$ is homeomorphic to a twisted interval bundle over $\Gamma$, contradicting the assumption on $R_i$.
This is a contradiction.
\end{proof}

We have the following finiteness statement for nested minimal surfaces:
\begin{prop}[Smooth Convergence of Nested Stable Minimal Surfaces]\label{nested}
Let $M$ be a three manifold with boundary $X$, a stable, minimal genus $g$ surface.  Suppose $M$ is not homeomorphic to a twisted interval bundle over a non-orientable surface. 
   Let $\{X_i\}_{i\in\mathbb{N}}$ be a sequence of stable minimal surfaces isotopic to $X$ so that each $X_i$ bounds $X\times [0,1]$ on one side and $H_i:=M\setminus (X\times [0,1])$ on the other.  Suppose $\{X_i\}_{i\in\mathcal{I}}$ are nested in the sense that 
\begin{equation}
H_i\subsetneq H_j\mbox{ whenever } i>j.  
\end{equation}
Then the areas of $X_i$ are uniformly bounded and thus some subsequence obtained from $\{X_i\}_{i\in\mathbf{N}}$ converges to a minimal surface of genus $g$ with a non-trivial Jacobi field.

If $M$ is endowed with a bumpy metric, no such infinite sequence $\{X_i\}_{i\in\mathbf{N}}$ can exist.
\end{prop}
The nestedness assumption is key in Proposition \ref{nested}.  Colding-Minicozzi have constructed \cite{CM} examples of stable tori without any area bound (later B. Dean \cite{D} found examples of any genus greater than $1$, and J. Kramer \cite{Kr} found examples of stable spheres).

\begin{rmk}
The proof of Proposition \ref{nested} is related to an idea due to M. Freedman and S.T. Yau toward proving the Poincare conjecture.  If one had a counterexample to the Poincare conjecture, the sketch was to endow it with a bumpy metric and then using an (as yet conjectural) min-max process to produce an infinite sequence of minimal embedded nested two-spheres.  Proposition \ref{nested} then leads to a contradiction.
\end{rmk}
\begin{proof}

Since $X_i$ do not bound a twisted $I$-bundle over a non-orientable surface, it follows from Lemma \ref{volumelowerbound} that there exists $\varepsilon_0>0$ so that $\mbox{vol}(H_i)\geq \varepsilon_0>0$ for all $i$.

For each $i\geq 1$, denote $R_i=H_{i+1}\setminus H_{i}$.   Note that as 
\begin{equation}
\bigcup_{i=1}^\infty R_i\subset H_1,
\end{equation}
is a disjoint decomposition, it follows that
\begin{equation}
\mbox{vol}(R_i)\rightarrow 0 \mbox{ as } r\rightarrow 0.
\end{equation}

For each $i\geq 1$, because $X_i$ and $X_{i+1}$ are nested we can define the lamination $\mathcal{L}_i$ to consist of two leafs, $X_i$ and $X_{i+1}$.  
By the curvature bounds for stable minimal surface \eqref{bounds}  it follows that $\mathcal{L}_i$ subconverge to a lamination $\mathcal{L}$.

From the definition of lamination, we can cover $M$ by finitely many open balls $\{B_i\}_{i=1}^p$ so that on any ball $B_i$, after applying a diffeomorphism $\phi^i:B_i\rightarrow D^2\times [0,1] \subset\mathbb{R}^3$, where $D^2$ is the unit disk in the $xy$-plane.

For $i$ large, in each ball $B_k$,  $\mathcal{L}_i|_{B_k}$ consists of several parallel graphs $\{G^{i,k}_j\}_{j=1}^{r}$, where $r$ depends on $i$ and $k$.

The region $R_i$ restricted to $B_k$ consists of several slab regions between consecutive graphs from $\{G^{i,k}_j\}_{j=1}^{r}$.   Let $R_i^k(s,s+1)$ denote 
such a slab region between $G^{i,k}_s$ and $G^{i,k}_{s+1}$ in $B_k$.  Slab regions come in three possible types: (a) both boundary walls $G^{i,k}_s$ and $G^{i,k}_{s+1}$ are in $X_i$, (b) one boundary wall is in $X_i$ and the other in $X_{i+1}$, or (c) both boundaries in $X_{i+1}$. 

For each slab $R_i^k(s,s+1)$ of type (b) there's a well defined retraction of $R_i^k(s,s+1)$ onto the wall in $X_i$.  For a slab $G^{i,k}_s$ of type (c), we can retract the slab onto a graph $\tilde{G}$ in between $G^{i,k}_s$ and $G^{i,k}_{s+1}$.

Note that there must be some points of $X_{i+1}$ which are retracted to $X_{i}$
(otherwise, all points of $X_{i+1}$ are of type (c) and thus we obtain that $X_{i+1}$ bounds a set of tiny volume, contradicting the previous lemma).

In this way, we obtain a smooth retraction of $X_{i+1}$ onto $X'_i\cup \tilde{X}$, where $\tilde{X}$ are the union of graphs arising in case (c) and $X'_i$ denotes a subset of $X_i$.  But since the retraction is smooth, in fact $\tilde{X}$ is empty and $X'_i=X_i$.

Thus we obtain that $X_{i+1}$ is a normal graph over $X_i$ for $i$ large.  In fact, the same argument shows that $X_j$ is a normal graph over $X_i$ for any $j>i$.  It follows that the areas of $X_i$ are uniformly bounded.  

\end{proof}

We can now prove the following Proposition:
\begin{prop}[Min-max in Handlebody]\label{handlebody}
Let $H$ be a Riemannian handlebody with strictly stable boundary, endowed with a bumpy metric.  Then there exists a minimal surface $\Sigma$ of index $1$ or $0$ obtained after finitely many neck-pinch surgeries performed on $\partial B$.   Thus the genus of $\Sigma$ is at most the genus of $\partial H$.
\end{prop}

\begin{proof}
We consider sweep-outs of $H$ beginning at $\partial H$ and ending at the one-dimensional spine of $H$.  We then consider the associated min-max problem and its width.  By Proposition \ref{boundary}, since $\partial H$ is strictly stable we obtain that the width of the handlebody $H$ is strictly larger than the area of $\partial H$. 

If $\partial H$ is diffeomorphic to $S^2$, then by the Deformation Theorem \ref{deformation}
the min-max limit can be realized by an unstable surface in the interior of $H$.
Otherwise, by the genus bounds, the minimal surface realizing the width cannot be equal to the boundary surface $\partial H$ obtained with some multiplicity.  Moreover, by the upper index estimates of Marques-Neves (Deformation Theorem A), we can guarantee that the min-max limit has index at most $1$.
\end{proof}

If we are in the case of obtaining a stable minimal surface $\Sigma$ in Proposition \ref{handlebody}, then we can apply Proposition \ref{handlebody} iteratively to the handlebody bounded by $\Sigma$ inside of $H$.  By the Nesting Proposition \ref{nested}, there can only be finitely many iterations and thus we obtain the following improvement on Proposition \ref{handlebody} which rules out obtaining a stable minimal surface:

\begin{prop}[Iterated Min-max in Handlebody]\label{handlebody2}
Let $H$ be a Riemannian handlebody with strictly stable boundary, endowed with a bumpy metric.  Then there exists a minimal surface $\Sigma$ of index $1$ obtained from $\partial H$ after finitely many neck-pinch surgeries performed.   Thus the genus of $\Sigma$ is at most the genus of $\partial H$.
\end{prop}

Note that Proposition \ref{handlebody2} implies the Index Bound Theorem \ref{indexstrong} as if we obtain a minimal surface that is strictly stable, Proposition \ref{handlebody2} implies the existence of an index $1$ surface inside each handlebody.

\subsection{Strong Irreducibility}

Let us recall some basic fact about strongly irreducible Heegaard splittings $\Sigma$.  

Let $\Sigma$ be a Heegaard surface in $M$.  By definition, this means $M\setminus\Sigma$ consists of two open handlebodies, $H_1$, and $H_2$.

Let $\gamma$ be a simple closed curve on $\Sigma$.  We say $\gamma$ is a \emph{compressing curve} if it bounds an embedded disk $D$ with $\partial D=\gamma$ whose interior is contained in $H_1$ or $H_2$.   We call $D$ the \emph{compressing disk} bounded by $\gamma$.   There are three types of compression curves:
\begin{enumerate}
\item $\gamma$ bounds no disk on $\Sigma$ and bounds a disk in $H_1$
\item $\gamma$ bounds no disk on $\Sigma$ and  bounds a disk in $H_2$
\item $\gamma$ bounds a disk in $\Sigma$ isotopic to its compressing disk in either $H_1$ or $H_2$
\end{enumerate}

In the third case, let us say that $\gamma$ bounds an \emph{inessential disk} and the compression is inessential.   In the first and second cases, let us say $\gamma$ bounds an \emph{essential} disk in $H_1$ or $H_2$, respectively..

A Heegaard surface is \emph{strongly irreducible} if every essential compressing disk in $H_1$ intersects every essential compressing disk in $H_2$.

Given a compressing curve on $\Sigma$, we can perform a neck-pinch surgery on $\Sigma$ along $\gamma$ to produce a new surface, $\Sigma'$.  Let us write $\Sigma'<\Sigma$ in this case.  Note that we have $\mbox{genus}(\Sigma')\leq\mbox{genus}(\Sigma)$.   In fact, whenever $\Sigma$ is the boundary of a handlebody and $\Sigma'<\Sigma$ precisely one of the following holds:
\begin{enumerate}
\item  $\mbox{genus}(\Sigma')=\mbox{genus}(\Sigma)-1$ 
\item  $\mbox{genus}(\Sigma')=\mbox{genus}(\Sigma)$ and the number of connected components of $\Sigma'$ is one more than $\Sigma$. 
\end{enumerate}
Because of strongly irreducibility, every compressing curve bounding a disk in $H_1$ intersects every such curve in $H_2$.  It follows that if we perform an $H_i$-compression (for $i=1$ or $i=2$) on $\Sigma$ to obtain a surface $\Sigma'$, then any further compression on $\Sigma'$ must be performed on the same side, except for inessential compressions which can happen on either side.  

Suppose we perform an essential neck-pinch on a strongly irreducible Heegaard splitting of genus $g$, $\Sigma_g$ to obtain $\Sigma_g'$.  It is possible that the genus of $\Sigma'_g$ is one less than $\Sigma_g$, and, while $\Sigma'_g$ bounds on one side a genus $g-1$ handlebody, on the other side it bounds a twisted interval bundle over a non-orientable surface $\Gamma$.  If such a neck-pinch exists, then $\Gamma$ is an incompressible non-orientable minimal surface.  Moreover, $\Gamma$ is known as a one-sided Heegaard splitting, since $M\setminus\Gamma$ is a handlebody of genus $g-1$.

The simplest example of this phenomenon arisees from the genus $1$ Heegaard splitting of $\mathbb{RP}^3$.  After a neck-pinch performed on a Heegaard torus, one obtains a three ball that bounds an interval bundle over $\mathbb{RP}^2$. 

Let us summarize this discussion in the following dichotomy for surgeries performed on a strongly irreducible splitting.  Suppose $\Sigma_k<\Sigma_{k-1}<...< \Sigma_0$ where $\Sigma_0$ is strongly irreducible.  Then one of the following holds:
\begin{enumerate}
\item  For each $j\in\{0,,,,k-1\}$, $\Sigma_j$ is obtained from an essential surgery on $\Sigma_{j-1}$ performed on the $H_1$ side or an inessential surgery splitting off a two-sphere
\item  For each $j\in\{0,,,,k-1\}$, $\Sigma_j$ is obtained from an essential surgery on $\Sigma_{j-1}$ performed on the $H_2$ side or an inessential surgery splitting off a two-sphere
\end{enumerate}

We have the following further dichotomy: 
\begin{enumerate}[label=(\alph*)]
\item $\Sigma_1$ bounds a twisted interval bundle over an incompressible one sided Heegaard surface and $\Sigma_j$ for $j>1$ are obtained from $\Sigma_{j-1}$ by inessential compressions or else
\item  for any $j\in\{1,2,...,k\}$  each non-sphere component of $\Sigma_j$ is incompressible in the manifold $M\setminus\mbox{int}(\Sigma_j)$, where $\mbox{int}(\Sigma_j)$ denotes the interior of the handlebody determined by $\Sigma_j$.   For any sphere component $\Sigma_j$, the infimal area of a surface isotopic to $\Sigma_j$ in $M\setminus\mbox{int}(\Sigma_j)$ is positive.
\end{enumerate}
If we are in case (a) we can minimize area in the isotopy class of the one sided Heegaard surface by a theorem of Meeks-Simon-Yau (\cite{MSY}) to obtain a stable embedded non-orientable minimal surface.  
\\
\\
\noindent

\emph{Proof of Theorems \ref{pr2} and \ref{lensspaces}:}
Let us first assume that $M$ is endowed with a bumpy metric.  So we are assuming $M$ is either hyperbolic or a lens space.  

Let $\Sigma$ be a strongly irreducible Heegaard splitting.  Suppose that after performing a neck-pinch on $H$, one obtains a surface isotopic to the boundary of a tubular neighborhood of a one-sided Heegaard splitting surface $\Gamma$.  Then we are in case (2) of the Theorem and by (a) above we can minimize in the isotopy class of $\Gamma$ to obtain a stable one-sided Heegaard splitting of $M$.

Let us therefore assume without loss of generality no such neck-pinch is possible. 

If $M$ is a hyperbolic manifold, there are no minimal tori or spheres and moreover, if $\Sigma_g$ is a minimal surface of genus $g$, then
\begin{equation}\label{aprioriareabound}
|\Sigma_g|< 4\pi(1-g).
\end{equation}

Consider the min-max limit $\Gamma_0$ obtained relative to $H$.  Thus we have
\begin{equation}
\Gamma_0=\sum_{i=1}^k n_i\Lambda_i, 
\end{equation}
where $\Lambda_i$ are closed embedded minimal surfaces and $n_i$ as positive integers.  Moreover, if $n_i>1$ then $\Lambda_i$ is a two-sphere.   Also we have
\begin{equation}
\sum_{i=1}^k n_i\mbox{genus}(\Lambda_i)\leq g. 
\end{equation}
From the Index bound \eqref{generic2} we obtain 
\begin{equation}
\sum_{i=1}^k\mbox{index}(\Lambda_i)=1
\end{equation}

Suppose that $\Lambda_1$ is the unique unstable component of $\Gamma_0$.   If $\Lambda_1$ has genus $g$, we are done.  Assume without loss of generality that the genus $g'$ of $\Lambda_1$ is less than $g$.  Thus from case (b) $\Lambda_1$ bounds a handlebody $Y_i$.  Let us remove from the manifold $M$ the set $Y_i$ to obtain a new manifold with boundary $M'=M\setminus Y_i$.  Since $\Lambda_1$ is unstable, we can minimize area in its isotopy class in $M'$ to obtain a closed embedded strictly stable minimal surface $\Lambda'_1$.  If $\Lambda_1$ has positive genus, then it is incompressible and $\Lambda'_1$ is a genus $g'$ strictly stable minimal surface in the isotopy class of $\Lambda_1$ together (potentially) with some minimal two-spheres.  If $\Lambda_1$ is a sphere, we obtain that $\Lambda_1'$ is a collection of minimal two-spheres.   If $\Lambda_1$ has positive genus, let us define $M''$ to be obtained from $M'$ be removing the collar neighborhood between the positive genus component of $\Lambda_1'$ and $\Lambda_1$.  If $\Lambda_1$ has genus zero, then one of the components of $\Lambda_1'$ is homologous to $\Lambda_1$, and thus we can form $M''$ be removing the cylindrical region between this component and $\Lambda_1$.   In the end, we obtain a manifold $M''$ with strictly stable boundary consisting of a single strictly stable minimal surface of genus $0\leq g'<g$.

Let us assume toward a contradiction that $M$ contains no index $1$ minimal surface isotopic to $\Sigma$. We will now describe an iteration process.  Beginning with $N_0=M''$, we obtain an infinite sequence of manifolds $\{N_i\}_{i=0}^\infty$ so that
\begin{enumerate}[label=(\alph*)]
\item $N_i$ has stable boundary $\partial N_i$ whose components are partitioned into two sets: $B^1_i$, and $B^2_i$, so that  for $j\in\{1,2\}$ each surface in $B^j_i$ is obtained from a sequence of $H^i$-surgeries on $\Sigma$. 
\item for each $i$ and $j\in\{1,2\}$ there holds $$\sum_{C\in B^j_i}\mbox{genus}(C)\leq g$$
\item $N_i\subsetneq N_j$ when $i>j$ 
\end{enumerate}

Note first that $N_0$ satisfies (1) and (2).  Given the surface $N_i$, let us describe how to construct $N_{i+1}$.

We can consider sweep-outs $\{\Sigma_t\}_{t=1}^1$ of $N_i$ with the following properties
\begin{enumerate}
\item $\Sigma_0 = B^1_i \cup \{\mbox{arcs joining the components of } B^1_t\}$
\item $\Sigma_t$ for $0<t<1$ is a surface isotopic to the strongly irreducible Heegaard surface $\Sigma$ 
\item $\Sigma_1 = B^2_i \cup \{\mbox{arcs joining the components of } B^2_t\}$
\item $\Sigma_t$ varies smoothly for $0<t<1$ and in the Hausdorff topology for $t\rightarrow 0$ and $t\rightarrow 0$
\end{enumerate}

We can consider the saturation of all such sweepouts $\Pi_{\Sigma_t}$ and width $W_i$ for the corresponding min-max problem.  Since the components of  
$B^2_i\cup B^1_i$ are strictly stable, it follows from Proposition \ref{boundary} that 
\begin{equation}
W_i>\sup_{C\in B^2_i\cup B^1_i} |C|. 
\end{equation}

Thus the min-max limit associated to the min-max problem consists of a varifold
\begin{equation}
\Lambda=\sum_{i=1}^k\Lambda_i.
\end{equation}

If $M$ is hyperbolic, then there are no minimal two-spheres and thus each component of $\Lambda$ arises with multiplicity $1$.  It then follows from Proposition \ref{boundary} that (at least) one component $\Lambda_1$ of $\Lambda$ is contained in the interior of the manifold. 
If instead $M$ is a lens space, then the Deformation Theorem similarly implies that a connected component of $\Lambda$ is contained in the interior of $M$.

  If $\Lambda_1$ is isotopic to one of the components of  $\partial N_i$,  then we remove from $N_i$ the collar region diffeomorphic to $\Lambda_1\times [0,1]$ to obtain a new manifold $N_i'$.   If $\Lambda_1$ bounds.a handlebody in $N_i$, then we remove this handlebody to obtain $N_i'$ (note in the process we may remove some components of $\partial N_i$ contained in this handlebody).   Since $N'_i$ now has one unstable component, $\Lambda_1$, we can minimize area \cite{MSY} in the isotopy class of $\Lambda_1$ in $N'_i$ to obtain a stable minimal surface $\tilde{\Lambda}_1$.  If $\Lambda_1$ has positive genus, then $\tilde{\Lambda}_1$ and $\Lambda_1$ are isotopic and we remove the collar region between them to obtain $N_{i+1}$.  If $\Lambda_1$ is a two-sphere, then it follows from our interpolation theorem that in fact $\tilde{\Lambda}_1$ consists of several two spheres with multiplicity $1$, one $\tilde{\Lambda}'_1$ of which is homologous to $\Lambda_1$.  We remove the collar region from $N_i$ between $\tilde{\Lambda}'_1$ and $\Lambda_1$ to obtain $N_{i+1}$.

This completes the construction of the iteration process to obtain an infinite sequence $N_1,N_2,...N_k...$ of manifolds satisfying (a), (b) and (c) above under the assumption that there is no index $1$ minimal surface isotopic to $\Sigma$.   We will now deduce a contradiction.

Since by (b) the genus of the surfaces $\partial N_i$ are bounded by $g$, if the areas of $\partial N_i$ were bounded, we obtain a Jacobi field, contradicting the bumpiness of the metric.   In the case that $M$ is hyperbolic, since $\partial N_i$ have bounded genus, it follows from the area bound \eqref{aprioriareabound} that the areas are in fact bounded, violating the bumpiness of the ambient metric. This is a contradiction and completes the proof in the hyperbolic case. 

Thus henceforth we may assume that the areas of $\partial N_i$ are an unbounded sequence of positive numbers and the ambient manifold is a lens space. 

Suppose for some subsequence of the $\partial N_i$, at least one component $\partial N^*_i$ of $\partial N_i$ that is distinct from $\partial N_{i-1}$ has genus $0<g'<g$.  In this case, we can assume that $\partial N_i$ are a nested sequence of genus $g'$ surfaces.   By Proposition \ref{nested} it follows that this case cannot happen as the metric is bumpy.  

Suppose instead for some subsequence of the $\partial N_i$, at least one component $\partial N^*_i$ of $\partial N_i$ distinct from  $\partial N_{i-1}$ is a sphere.  We claim that upon passing to a subsequence, we can assume that these spheres $\partial N_i$ are nested.  If not, by Lemma \ref{volumelowerbound}, each two-sphere bounds a definite volume $\mbox{vol}$.  If the two-spheres $\partial N_i$ are non-nested, then we obtain an impossibility:
\begin{equation}
\mbox{vol}(M)\geq\sum_{i=1}^\infty\mbox{vol}(\mbox{int}(\partial N^*_i))\geq \sum_{i=1}^\infty \mbox{vol}=\infty,
\end{equation}
where $\mbox{int}(\partial N^*_i)$ denotes the three-ball in $M$ enclosed by the two-sphere $\partial N^*_i$.

Thus we can assume the two spheres $\partial N^*_i$ are nested.  By Proposition \ref{nested}, we obtain a contradiction. 

Since all cases result in contradictions, it follows that the sequence $N_0, N_1, N_2...$ must terminate at a finite stage so that $N_k'$ is a surface isotopic to $\Sigma$ with Morse index $1$.

If the metric $g$ is not bumpy, we can consider a sequence of metrics $g_i\rightarrow g$.  By the above, there is an index $1$ minimal surface $\Sigma_i$ isotopic to $\Sigma$.  We can consider the limit $\Sigma_i\rightarrow\Sigma_\infty$. Since $\Sigma_i$ have bounded area and bounded genus, $\Sigma_\infty$ is a smooth connected minimal surface obtained with some multiplicity.   By strong irreducibility, the converge is multiplicity $1$ (as the only other alternative is for it converge with multiplicity two to one sided Heegaard splitting surface).  Since the convergence is multiplicity $1$ it is smooth everywhere, and thus the genus of $\Sigma$ is the same as $\Sigma_i$ and moreover, $\Sigma$ is isotopic to $\Sigma_i$.  
The index of $\Sigma_\infty$ is either $0$ or $1$ because under smooth convergence, the eigenvalues of the stability operator vary smoothly.
\qed


Finally, let us prove Theorem \ref{rp3}, which we restate:
\begin{thm}
Let $M$ be a Riemannian $3$-manifold diffeomorphic to $\mathbb{RP}^3$ endowed with a bumpy metric.  Then $M$ contains a minimal index $1$ two-sphere or minimal index $1$ torus.  
 \end{thm}

 \begin{proof}
 If the minimal surface realizing the width of $\mathbb{RP}^3$ contains a two-sphere or torus, then we are done.  Otherwise, the width is realized by $k\Gamma$, where $\Gamma$ is an embedded $\mathbb{RP}^2$ and $k$ is an even integer.  
 
 Let us pass to a double cover $\tilde{M}$ of $M$ so that $\Gamma$ lifts to $\tilde{\Gamma}$ which a two-sphere and $\tilde{M}$ is a three-sphere.   If $\tilde{\Gamma}$ is strictly stable, then by our Index Bound, each ball $\tilde{B}_1$, $\tilde{B}_2$ bounded by $\tilde{\Gamma}$ contains an unstable two-sphere $\tilde{S}_1$ and $\tilde{S}_2$, respectively. 
 
  Since $\tilde{S}_1$ is contained in a fundamental domain of the deck group, it follows that $\tilde{S}_1$ descends to a minimal two-sphere in $M$.  If instead $\tilde{\Gamma}$ is unstable, then we can push off $\tilde{\Gamma}$ to one side of $\tilde{\Gamma}$ using the lowest eigenfunction of the stability operator.  Using a by-now standard argument (c.f. Lemma 3.5 in \cite{mn1}) the sweepout can either be extended to the rest of $\tilde{B}_1$ with all areas below $\tilde{\Gamma}$ or else we obtain an unstable two-sphere in $\tilde{B}_1$, which means the theorem is proved as the unstable two-sphere in $\tilde{B}$ descends to $M$.    Thus we can assume we have realized $\tilde{\Gamma}$ as a minimal surface in an optimal foliation of $M$.  By the catenoid estimate \cite{KMN} we can easily construct a sweep-out of $M$ by tori with area less than $2|\Gamma|$.  Thus the width of $M$ could not have been realized by $k\Gamma$.  
 
 \end{proof} 

 This generalizes in a straightforward way to the case of strongly irreducible Heegaard splittings:
 
 \begin{thm}\label{orientable}
 Let $M$ be a hyperbolic three-manifold and $\Sigma$ a strongly irreducible Heegaard splitting of genus $g$.  Then $M$ contains an orientable index $1$ minimal surface of genus at most $g$.
 \end{thm}
 
  \begin{proof}
  The proof is identical to the case of $\mathbb{RP}^3$ above.  By Pitts-Rubinstein claim, either $M$ contains a genus $g$ minimal surface isotopic to $\Sigma$, or $M$ contains a one sided minimal Heegaard splitting $\Gamma$.   In the first case, the theorem is proved, so we can assume $M$ contains a non-orientable minimal surface $\Gamma$.  
  
There is a double cover (see Section 2 of \cite{R2}) of the manifold $M$ so that the one-sided Heegaard surface $\Gamma$ lifts to become a Heegaard surface $\tilde{\Gamma}$ of genus $g$.   The proof is then identical to the case of $\mathbb{RP}^3$ above.   \end{proof}
  
 

   \end{document}